\documentclass[11pt]{article}

\usepackage[utf8]{inputenc}
\usepackage[T1]{fontenc}

\usepackage{amsthm, amsmath, amsfonts, amssymb, bm, mathtools}

\usepackage[authoryear]{natbib}

\usepackage{graphicx}
\usepackage[a-2u]{pdfx}
\usepackage[inline]{enumitem}
\usepackage{fullpage}
\usepackage{hyperref}
\usepackage{orcidlink}

\theoremstyle{plain}
\newtheorem{theorem}{Theorem}
\newtheorem{lemma}[theorem]{Lemma}
\theoremstyle{definition}

\newtheorem{example}[theorem]{Example}
\newtheorem{note}{Remark}

\newcommand{\D}{\mathcal{D}}
\newcommand{\SHD}{\mathcal{SD}}           

\newcommand{\Smed}{\mathfrak{M}^\mathrm{sc}}        
\newcommand{\Lmed}{\mathfrak{M}^{\mathrm{loc}}}        
\newcommand{\Sl}[1]{\mathrm{Sl}\left(#1\right)} 
\newcommand{\cSl}[1]{\mathrm{cSl}\left(#1\right)} 
\newcommand{\trace}{\mathrm{tr}} 
\newcommand{\diag}{\mathrm{diag}} 
\newcommand{\R}{\mathbb{R}}              
\newcommand{\N}{\mathbb{N}}                 
\newcommand{\set}[1]{\left\lbrace #1 \right\rbrace} 
\newcommand{\PD}[1][d]{\mathbb{PD}_{#1}}       
\renewcommand{\P}[1]{\mathcal{P}(\R^{#1})}  
\renewcommand{\S}[1]{\mathbb{S}^{#1-1}}     
\newcommand{\prob}{\mathbb{P}}              
\newcommand{\E}{\mathbb{E}}                 
\newcommand{\eqdis}{\overset{d}{=}}                      
\newcommand{\T}{^\mathsf{T}}                             
\newcommand{\abs}[1]{\left|{#1}\right|}                  
\renewcommand{\vec}[1]{{\bm{#1}}}                        
\newcommand{\mat}[1]{\mathbf{#1}}                        
\newcommand{\inner}[1]{\left\langle#1\right\rangle}      
\newcommand{\norm}[1]{\left\lVert#1\right\rVert}         

\newcommand{\hs}{^\mathrm{hs}}
\newcommand{\dist}[1][\alpha]{\mathfrak{D}_{#1}}

\newlength{\leftstackrelawd}
\newlength{\leftstackrelbwd}
\def\leftstackrel#1#2{\settowidth{\leftstackrelawd}%
{${{}^{#1}}$}\settowidth{\leftstackrelbwd}{$#2$}%
\addtolength{\leftstackrelawd}{-\leftstackrelbwd}%
\leavevmode\ifthenelse{\lengthtest{\leftstackrelawd>0pt}}%
{\kern-.5\leftstackrelawd}{}\mathrel{\mathop{#2}\limits^{#1}}}

\makeatletter
\newcommand{\proofpart}[2]{%
  \par
  \addvspace{\medskipamount}%
  \noindent\emph{\textbf{Part #1: #2.}}
}
\makeatother

\title{Location and scatter halfspace median under $\alpha$-symmetric distributions}

\author{Filip Bo\v{c}inec\,\orcidlink{0009-0000-4415-9302}\thanks{Faculty of Mathematics and Physics, Charles University, Prague, Czech Republic. Email: \texttt{bocinec@karlin.mff.cuni.cz}} 
\and Stanislav Nagy\,\orcidlink{0000-0002-8610-4227}\thanks{Faculty of Mathematics and Physics, Charles University, Prague, Czech Republic. Email: \texttt{nagy@karlin.mff.cuni.cz}}}

\date{}  

\begin{document}

\maketitle

\begin{center}
\textit{Preprint (Author’s Original Manuscript). This is the original 
unrefereed version of the manuscript. The Version of Record has been 
published in the \textit{Journal of Nonparametric Statistics} 
(\textcopyright\ Taylor \& Francis) and is available at: 
\url{https://doi.org/10.1080/10485252.2025.2600417}.}
\end{center}

\begin{abstract}
In a landmark result, \citet{Chen_etal2018} showed that multivariate medians induced by halfspace depth attain the minimax optimal convergence rate under Huber contamination and elliptical symmetry, for both location and scatter estimation. We extend some of these findings to the broader family of $\alpha$-symmetric distributions, which includes both elliptically symmetric and multivariate heavy-tailed distributions. For location estimation, we establish an upper bound on the estimation error of the location halfspace median under the Huber contamination model. An analogous result for the standard scatter halfspace median matrix is feasible only under the assumption of elliptical symmetry, as ellipticity is deeply embedded in the definition of scatter halfspace depth. To address this limitation, we propose a modified scatter halfspace depth that better accommodates $\alpha$-symmetric distributions, and derive an upper bound for the corresponding $\alpha$-scatter median matrix. Additionally, we identify several key properties of scatter halfspace depth for $\alpha$-symmetric distributions.
\end{abstract}

\vspace{1em}

\noindent\textbf{MSC classes:} Primary 62G05; 62G35; 62H12

\vspace{0.5em}


\section{Introduction: Location and scatter halfspace median} 

Robust estimation of location and scale for univariate data has been one of the~cornerstones of robust statistics, and is currently already well understood \citep{Huber_Ronchetti2009, Hampel_etal1986}. Among location estimators, a~prominent place is occupied by the~median, naturally generating high breakdown estimators satisfying plausible equivariance properties. Its scale counterpart is the~median absolute deviation, sharing a~similar array of desirable traits.

Robust estimation of location and scatter for multidimensional data is much more involved, and no canonical high-breakdown equivariant analogs of the~median or the~median absolute deviation exist. Instead, a~variety of diverse approaches, each with its own advantages and limitations, can be found in the~literature. Selecting from the~more recent approaches, we refer to \citet{Rousseeuw_Hubert2013, Maronna_Yohai2017, Lugosi_Mendelson2021, Dalalyan_Miansyan2022, Zhang_etal2024} and \citet{Fishbone_Mili2024}.

In this paper, we focus on two outstanding location and scatter estimators induced by the~halfspace depth for location and scatter, respectively. Pioneered by \citet{Tukey1975} and introduced to robust statistics by \citet{Donoho_Gasko1992}, the~halfspace depth is a~well-studied tool of nonparametric statistics whose aim is to establish concepts such as ordering, ranks, or quantiles to multivariate datasets. Writing $\P{d}$ for the~set of all Borel probability distributions on $\R^d$, the~\emph{halfspace depth} (abbreviated as \textbf{HD}) of a~point $\vec{x} \in \R^d$ with respect to (w.r.t.) $P \in \P{d}$ is defined as
    \begin{equation}    \label{eq: halfspace depth}
    \D(\vec{x}; P) = \inf_{\vec{u}\in\S{d}} \prob\left(\inner{\vec{X}, \vec{u}}\leq \inner{\vec{x}, \vec{u}}\right),
    \end{equation}
where $\vec{X}\sim P$ and $\S{d} = \left\{ \vec{u} \in \R^d \colon \norm{\vec{u}}_2^2 = \inner{\vec{u}, \vec{u}} = 1 \right\}$ is the~unit sphere in $\R^d$. The~\textbf{HD} is also called Tukey, or location depth. It quantifies the~centrality of $\vec{x}$ within the~geometry of the~mass of $P$. The~higher $\D(\vec{x}; P)$ is, the~more `representative' the~point $\vec{x}$ is of the~location of $P$. The~\textbf{HD}~\eqref{eq: halfspace depth} induces a~natural location parameter called the~halfspace median (also Tukey's median), defined as the~barycenter\footnote{A barycenter of a~non-empty compact convex set $S \subset \R^d$ is the~expectation of a~random vector that is uniformly distributed on $S$. The~set $\Lmed(P)$ is always non-empty, compact, and convex as proved in \citet[Section~3]{Rousseeuw_Ruts1999}.} $\vec{\mu}\hs = \vec{\mu}\hs(P) \in \R^d$ of the~halfspace median set
    \begin{equation}    \label{eq: Tukey median}
    \Lmed(P) = \left\{ \vec{x} \in \R^d \colon \D(\vec{x}; 
 P) = \max_{\vec{y} \in \R^d} \D(\vec{y}; P) \right\}.
    \end{equation}
For $d = 1$, $\vec{\mu}\hs$ is the~classical median. For a~random sample $\set{\vec{X}_i}_{i=1}^n$ from $P \in \P{d}$ and the~associated empirical distribution $\widehat{P}_n \in \P{d}$ assigning mass $1/n$ to each $\vec{X}_i$, $i = 1, \dots, n$, we define the~sample \textbf{HD} of $\vec{x}\in\R^d$ w.r.t. $\widehat{P}_n$ as $\D(\vec{x}; \widehat{P}_n)$. Its deepest point $\widehat{\vec{\mu}}\hs_n = \vec{\mu}\hs\left( \widehat{P}_n\right)$ is called the~\emph{sample halfspace median} of $\set{\vec{X}_i}_{i=1}^n$.

A robust estimator of the~scatter parameter of $P \in \P{d}$ that shares similarities with the~halfspace median is obtained when using the~scatter halfspace depth~\citep{Zhang2002, Chen_etal2018, Paindaveine2018}. Denote by $\PD$ the~set of all positive definite $d\times d$ matrices. The~\emph{scatter halfspace depth} of $\mat{\Sigma} \in \PD$ w.r.t. $P \in \P{d}$ is defined as
    \begin{equation}    \label{eq: SHD}
    \SHD (\mat{\Sigma}; P) 
     = \inf_{\vec{u} \in \S{d}} \min 
    \left\{ 
    \prob\left( \abs{\inner{\vec{X}-T(P), \vec{u}}} \leq \sqrt{\vec{u}\T \mat{\Sigma} \vec{u}} \right),\prob\left( \abs{\inner{\vec{X}-T(P), \vec{u}}} \geq \sqrt{\vec{u}\T \mat{\Sigma} \vec{u}} \right) 
    \right\}, 
    \end{equation}
where $T\colon \P{d}\to\R^d$ is a~properly chosen location functional. Throughout this paper, we take $T(P)$ to be the~halfspace median $\vec{\mu}\hs$ of $P$.\footnote{This choice is made only for convenience; all our results hold true for any other affine equivariant location functional $T$, using obvious minor modifications.} The~scatter halfspace depth (abbreviated as \textbf{sHD}) measures the~`centrality' of a~candidate scatter matrix $\mat{\Sigma}$ within the~space $\PD$ w.r.t. $P$. As for the~location case, consider the~set
    \begin{equation} \label{eq: scatter median set} 
    \Smed(P) = \left\{ \mat{\Sigma} \in \PD \colon \SHD(\mat{\Sigma}; P) = \sup_{\mat{S} \in \PD} \SHD(\mat{S}; P) \right\} 
    \end{equation}
of matrices that maximize~\eqref{eq: SHD}. Note that unlike for the~location case, the~set $\Smed(P)$ can be empty in general~\citep[Section~3]{Paindaveine2018}. However, for $P\in\P{d}$ that is smooth at $T(P)$,\footnote{We say that $P \in \P{d}$ is smooth at $\vec{x}\in\R^d$ if each hyperplane $h$ passing through $\vec{x}$ has $P(h) = 0$. We say that $P$ is smooth if it is smooth at each point $\vec{x}\in\R^d$.} the~barycenter of $\Smed(P)$ is well defined~\citep[Section~3 and Theorem~4.3]{Paindaveine2018}. We call that barycenter the~\emph{scatter halfspace median matrix} $\mat{\Sigma}\hs = \mat{\Sigma}\hs(P) \in\PD$. The~matrix $\mat{\Sigma}\hs$ offers a~well-performing nonparametric alternative to the~usual scatter estimators.

Same as for the~location \textbf{HD}, the~sample \textbf{sHD} is defined by $\SHD(\mat{\Sigma}; \widehat{P}_n)$, and the~sample scatter halfspace median matrix is $\widehat{\mat{\Sigma}}\hs_n = \mat{\Sigma}\hs(\widehat{P}_n) \in \PD$. The~sample \textbf{sHD} attains finitely many values in $\set{i/n \colon i = 0, 1, \ldots, n}$, hence $\widehat{\mat{\Sigma}}\hs_n$ always exists.

This paper builds on the~recent remarkable result of \citet{Chen_etal2018}, who demonstrated that in the~classical Huber's $\varepsilon$-contamination model, the~halfspace median $\vec{\mu}\hs \in \R^d$ and the~scatter halfspace median matrix $\mat{\Sigma}\hs \in \PD$ both are minimax optimal under elliptical models. Revisiting and expanding the~proofs of \citet{Chen_etal2018}, our aim is to generalize their results on the~performance of the~multivariate medians to the~broader collection of the~so-called $\alpha$-symmetric distributions \citep{Misiewicz1996, Uchaikin_Zolotarev1999, Koldobsky2005} for $\alpha > 0$. That family provides a~versatile generalization of elliptical models (for $\alpha = 2$), also encompassing multivariate stable distributions, and a~broad spectrum of heavy-tailed distributions. The~$\alpha$-symmetric distributions have found many applications in engineering, finance, or physics. At the~same time, their plausible analytical properties make them the~largest general class of multivariate distributions whose \textbf{HD} and \textbf{sHD} are possible to be expressed explicitly. That was noted by \citet{Masse_Theodorescu1994} and \citet{Chen_Tyler2004} for the~\textbf{HD}, and used by \citet{Nagy2019} for the~\textbf{sHD}.

After introducing notations, in Section~\ref{sec:2}, we provide a~brief overview of the~$\alpha$-symmetric distributions and their properties. The~concentration inequality for the~location halfspace median $\vec{\mu}\hs$ for contaminated $\alpha$-symmetric distributions is derived in Section~\ref{sec:3}. In Section~\ref{sec:4}, we derive several results about the~scatter halfspace median matrix for $\alpha$-symmetric distributions. In particular, we show that under the~assumption of the~$\alpha$-symmetry of $P$, (i) the~scatter median set $\Smed(P)$ from~\eqref{eq: scatter median set} is a~singleton, (ii) its unique element $\mat{\Sigma}\hs$ is Fisher consistent, (iii) we derive the~explicit form of $\mat{\Sigma}\hs$, and (iv) show that this matrix $\mat{\Sigma}\hs(P)$ is continuous in the~argument of the~distribution $P$; in particular, it is always estimated consistently by $\widehat{\mat{\Sigma}}_n\hs$. The~problem of establishing a~concentration inequality for the~scatter halfspace median matrix under Huber's contamination model is treated in Section~\ref{sec:5}. First, in Section~\ref{section: rate for elliptical}, we present an upper bound on deviation of the~\textbf{sHD} of $\mat{\Sigma}\hs$ and recover the~upper bound for estimating the~scatter parameter of spherical distributions~\citep[Theorem~3.1]{Chen_etal2018} using the~scatter halfspace median. However, this method cannot be directly applied when $\alpha \neq 2$. To overcome this limitation, in Section~\ref{section: alphaSHD}, we define the~$\alpha$-scatter halfspace depth (abbreviated as $\alpha$-\textbf{sHD}), a~version of~\eqref{eq: SHD} adapted specifically to $\alpha$-symmetric distributions. Using the~median matrix induced by the~$\alpha$-\textbf{sHD}, a~concentration inequality analogous to that obtained for $\alpha = 2$ can be given. Additional minor technical details are collected in the~online Supplementary Material.

\smallskip\noindent\textbf{Notations.} The~set of positive integers is $\N$. The~elements of a~vector $\vec{x} \in \R^d$ are typically denoted by $\vec{x} = \left(x_1, \dots, x_d \right)\T$; the~elements of a~matrix $\mat{\Sigma} \in \R^{d \times d}$ can be written as $\mat{\Sigma} = \left( \sigma_{i,j} \right)_{i,j=1}^d$. By $\mat{I}$, we mean any square identity matrix. The~operator norm of a~symmetric matrix $\mat{A} \in \R^{d\times d}$ is defined by $\norm{\mat{A}}_{\mathrm{op}} = \sup_{\vec{u} \in \S{d}} \abs{\vec{u}\T \mat{A} \vec{u}}$. The~maximum of $a, b$ is denoted by $a\vee b$. An absolute constant $C$ means that the~constant $C$ does not depend on sample size $n$, dimension $d$, contamination amount $\varepsilon$, and confidence parameter $\delta$. Such absolute constants, typically denoted by $C, C_1, C_2$, etc., take different values in each result below.

All random quantities are defined on a~common probability space $\left( \Omega, \mathcal A, \prob\right)$. For $P \in\P{d}$, $\vec{X} \sim P$ means that the~random vector $\vec{X}$ has distribution $P$. We write $\vec{X} \eqdis \vec{Y}$ if the~random vectors $\vec{X}$ and $\vec{Y}$ have the~same distribution. For a~transform $\varphi \colon \R^d \to \R^k$ and $\vec{X} \sim P \in \P{d}$, we write $P_{\varphi(\vec{X})} \in \P{k}$ for the~distribution of the~transformed variable $\varphi(\vec{X}) \in \R^k$. In particular, $P_{\vec{X}} = P$. We say that $P \in \P{d}$ is smooth if 
    \begin{equation}\label{eq7}
    P\left(\set{\vec{x}\in\R^d\colon \inner{\vec{x}, \vec{u}}=t}\right)=0 \quad \mbox{for all $\vec{u}\in\S{d}$ and $t\in\R$}.
    \end{equation}

\section{Preliminaries on \texorpdfstring{$\alpha$}{alpha}-symmetric distributions}\label{sec:2}

For $\alpha > 0$, the~$\alpha$-norm\footnote{For $\alpha<1$, this function violates the~triangle inequality. Since we do not make use of the~triangle inequality of a~norm, we still call it a~norm for convenience.} of a~vector $\vec{x}=(x_1, \ldots, x_d)\T\in\R^d$ is defined as
\begin{equation*}
    \norm{\vec{x}}_\alpha = 
    \begin{cases}
        \left(\sum_{i=1}^d \abs{x_i}^\alpha\right)^{1/\alpha}&\text{ if }0<\alpha<\infty,\\
        \max\set{\abs{x_1}, \ldots, \abs{x_d}}&\text{ if }\alpha = \infty.
    \end{cases}
\end{equation*}
The distribution $P$ of a~random vector $\vec{X}=(X_1, \ldots, X_d)\T\sim P \in \P{d}$ is said to be \emph{$\alpha$-symmetric} \citep{Fang_etal1990, Misiewicz1996, Koldobsky2005} if the~characteristic function of $\vec{X}$ takes the~form
\begin{equation}\label{eq6}
    \psi_\vec{X}(\vec{t})=\E \exp{(i \inner{\vec{t}, \vec{X}})}=\phi(\norm{\vec{t}}_\alpha)\text{ for all }\vec{t}\in\R^d,
\end{equation}
where $\phi$ is a~continuous function on $\R$. An equivalent definition \citep[Theorem~7.1]{Fang_etal1990} is that for all $\vec{u}\in\S{d}$ it holds
\begin{equation}\label{eq1}
    \inner{\vec{X}, \vec{u}}\eqdis \norm{\vec{u}}_\alpha X_1
\end{equation}
for $X_1$ the~first element of $\vec{X}$. Because the~characteristic function~\eqref{eq6} is real and symmetric, $\vec{X}$ has to be centrally symmetric about the~origin in the~sense that $\vec{X}\eqdis -\vec{X}$. Consider the~univariate distribution function
    \begin{equation}\label{eq2}
    F(t) = \prob(X_1\leq t) \quad \text{for all }t \in\R.
    \end{equation}
Throughout this paper, we assume that for all $\alpha$-symmetric distributions $P \in \P{d}$ 
    \begin{enumerate}[ref=($\mathsf{A}_1$), label=\upshape{($\mathsf{A}_1$)}]
    \item \label{A1} it holds that $\prob(\vec{X}=\vec{0})=0$ for $\vec{X} \sim P$, and $d>1$.
    \end{enumerate} 
This assumption is imposed to avoid trivial and non-interesting situations. Condition~\ref{A1} implies that $P$ is smooth as in~\eqref{eq7}. That was proved by~\citet[Theorem~II.2.3]{Misiewicz1996}. In particular, the~function $F$ from~\eqref{eq2} is continuous on $\R$, $F(t)= 1-F(-t)$ for all $t\in\R$, and $F(0)=1/2$.

For $\alpha = 2$, the~collection of $\alpha$-symmetric distributions is exactly the~family of spherically symmetric distributions $P \in \P{d}$ \citep[Chapter~2]{Fang_etal1990}, characterized by the~property that for $\vec{X} \sim P$ and any $\mat{A} \in \R^{d \times d}$ orthogonal, $\mat{A} \vec{X} \eqdis \vec{X}$. An important example is the~standard $d$-variate Gaussian distribution, which is spherically symmetric.

For $\alpha \in (0,2)$, the~$\alpha$-symmetric distributions provide a~rich family of multivariate models with many important applications. In particular, they include multivariate stable distributions \citep{Uchaikin_Zolotarev1999}, one of the~most important classes of distributions in probability theory. For $\alpha > 2$, $\alpha$-symmetric distributions exist only in the~plane \citep[(P9) and the~discussion in Section~II.4]{Misiewicz1996}. In our general treatment of $\alpha$-symmetric distributions below we, however, cover also the~case $\alpha > 2$ and $d = 2$. A~particularly simple $1$-symmetric distribution is the~multivariate extension of the~Cauchy distribution described in the~following example.

\begin{example}\label{ex: 1}
Let $\vec{X}\sim P\in\P{d}$ consist of $d$ independent Cauchy distributed random variables. Since the~characteristic function of a~standard Cauchy random variable is $\exp(-\abs{t})$ for $t \in \R$, we have $\psi_\vec{X}(\vec{t})=\exp(-\sum_{j=1}^d \abs{t_j})=\exp(-\norm{\vec{t}}_1)$ for $\vec{t}\in\R^d$, and $\vec{X}$ is $1$-symmetric. The~distribution function of $X_1$ is $F(t)=1/2+\arctan(t)/\pi$ for $t\in\R$.
\end{example}

An important property of $\alpha$-symmetric distributions is their invariance w.r.t. the~group of symmetries of a~(hyper-)cube in $\R^d$. Such symmetries are represented by \emph{signed permutation matrices}, which are defined as matrices $\mat{A} \in \R^{d \times d}$ that have a~unique non-zero element in each row and each column, and each of these non-zero elements is either $1$ or $-1$. The~following lemma will be used in Section~\ref{sec:4}.

\begin{lemma}   \label{lemma: sign permutation}
Let $\vec{X}\sim P\in\P{d}$ be $\alpha$-symmetric. If $\mat{A} \in \R^{d\times d}$ is a~signed permutation matrix, then $\mat{A}\vec{X}\eqdis\vec{X}$.
\end{lemma}

\begin{proof}
For any $\vec{t}\in\R^d$, the~characteristic function of $\mat{A}\vec{X}$ is by~\eqref{eq6}
    \begin{equation*}
        \psi_{\mat{A}\vec{X}} (\vec{t}) =\E\exp(i\inner{\vec{t}, \mat{A}\vec{X}})=\E\exp\left(i\inner{\mat{A}\T\vec{t}, \vec{X}}\right) =\phi\left(\norm{\mat{A}\T\vec{t}}_\alpha\right)\leftstackrel{\mathrm{(A)}}{=}\phi\left(\norm{\vec{t}}_\alpha\right)=\psi_{\vec{X}}(\vec{t}).
    \end{equation*}
In $\mathrm{(A)}$, we used the~fact that $\mat{A}\T$ is also a~signed permutation matrix. Transforming $\vec{t}$ to $\mat{A}\T \vec{t}$ thus only permutes and possibly changes the~signs of the~entries of $\vec{t}$, leaving its $\alpha$-norm intact. Characteristic functions of $\mat{A}\vec{X}$ and $\vec{X}$ are equal, hence also $\mat{A}\vec{X}\eqdis\vec{X}$.
\end{proof}

We conclude this section with an important note on how to understand all the~results stated in this paper.

\begin{note}    \label{remark: affine}
Every $\alpha$-symmetric random vector $\vec{X} \sim P \in \P{d}$ defines a~location-scatter family of distributions
    \[  \mathcal{P}^{\mathrm{loc/sc}}(P) = \left\{ \mat{A}\vec{X}+\vec{\mu} \sim P_{\mat{A}\vec{X}+\vec{\mu}} \mbox{ for $\vec{\mu} \in \R^d$ and $\mat{A} \in \R^{d\times d}$ non-singular} \right\}
    \]
in $\P{d}$, given by the~distributions of random vectors $\mat{A}\vec{X}+\vec{\mu}$. Here, $\vec{\mu}\in\R^d$ is a~location parameter, and $\mat{A} \in \R^{d\times d}$ is a~non-singular matrix giving the~scatter parameter within $\mathcal{P}^{\mathrm{loc/sc}}(P)$. Using this construction for spherically symmetric distributions ($\alpha = 2$), we recover the~collection of all (full-dimensional) elliptically symmetric distributions \citep{Fang_etal1990}. That family, of course, contains all $d$-variate Gaussian distributions, and many more well-studied measures.

Throughout this paper, we state concentration inequalities for the~canonical location parameter $\vec{\mu}\hs = \vec{0} \in \R^d$ and canonical scatter parameter $\mat{\Sigma}\hs = \sigma^2 \mat{I} \in \PD$ (for appropriate $\sigma > 0$, see Section~\ref{sec:4}) for $\alpha$-symmetric random vectors $\vec{X} \sim P \in \P{d}$. All these results should be understood in the~context of the~estimation of location and scatter in the~induced location-scatter families $\mathcal{P}^{\mathrm{loc/sc}}(P)$. That can be seen because the~location halfspace median $\vec{\mu}\hs = \vec{\mu}\hs(P) \in \R^d$ \citep[Lemma~2.1]{Donoho_Gasko1992} and the~scatter halfspace median matrix $\mat{\Sigma}\hs = \mat{\Sigma}\hs(P) \in \PD$ \citep[Theorem~2.1]{Paindaveine2018} are affine equivariant, meaning that for any $P = P_{\vec{X}} \in \P{d}$, $\vec{\mu} \in \R^d$, and $\mat{A} \in \R^{d\times d}$ non-singular we have
    \begin{equation}   \label{eq: affine equivariance}
    \begin{aligned}
    \vec{\mu}\hs(P_{\mat{A}\vec{X} + \vec{\mu}}) & = \mat{A} \,\vec{\mu}\hs(P_{\vec{X}}) + \vec{\mu}, \\
    \mat{\Sigma}\hs(P_{\mat{A}\vec{X} + \vec{\mu}}) & = \mat{A} \, \mat{\Sigma}\hs(P_{\vec{X}}) \, \mat{A}\T.
    \end{aligned}
    \end{equation}
Applying affine equivariance, the~results derived throughout this paper for the~canonical location-scatter parameters for $\alpha$-symmetric distributions can be understood as results for the~location-scatter families $\mathcal{P}^{\mathrm{loc/sc}}(P)$ of $\alpha$-symmetric measures; for a~detailed argument in the~location case see also Remark~\ref{remark: affine 2} below.
\end{note}

\section{Estimation of location halfspace median under contamination}\label{sec:3}

Consider an $\alpha$-symmetric random vector $\vec{X} \sim P \in \P{d}$ whose first marginal is given by $F$ from~\eqref{eq2}. The~\textbf{HD} of a~point $\vec{x}\in\R^d$ w.r.t. $P$ can be derived explicitly
\begin{align}
    \begin{split}\label{eq3}
    \D(\vec{x}; P) &=\inf_{\vec{u}\in\S{d}}\prob\left(\inner{\vec{X}, \vec{u}}\leq \inner{\vec{x}, \vec{u}}\right) \leftstackrel{\eqref{eq1}}{=}\inf_{\vec{u}\in\S{d}}\prob\left(\norm{\vec{u}}_\alpha X_1\leq \inner{\vec{x}, \vec{u}}\right)\\
    &\leftstackrel{\eqref{eq2}}{=} 
    F\left(\inf_{\vec{u}\in\S{d}}
    \frac{\inner{\vec{x}, \vec{u}}}{\norm{\vec{u}}_\alpha}\right)  \leftstackrel{\mathrm{(H)}}{=}F\left(-\norm{\vec{x}}_\beta\right)=1-F\left(\norm{\vec{x}}_\beta\right),
    \end{split}
\end{align}
for $\beta$ the~conjugate index to $\alpha$ given by
\begin{equation}\label{eq: betaDef}
    \beta =
    \begin{cases}
        \frac{\alpha}{\alpha-1} & \text{if }\alpha>1, \\
        \infty & \text{if }0<\alpha\leq 1.
    \end{cases}
\end{equation}
Equality~$\mathrm{(H)}$ comes from a~generalized H\"older inequality; a~proof can be found in \citet[Lemma~5.1]{Chen_Tyler2004}. Under~\ref{A1}, \citet[Theorem~II.2.2, for $\alpha \ne 2$]{Misiewicz1996} and \citet[Theorem~2.10, for $\alpha = 2$]{Fang_etal1990} prove that all marginal distributions~\eqref{eq1} of any $\alpha$-symmetric $P \in \P{d}$ have connected support, meaning that $F$ is continuous and strictly increasing at $0$. That implies that the~unique halfspace median of any such $P$ is $\vec{\mu}\hs = \vec{0} \in \R^d$, and the~maximum \textbf{HD} is $\sup_{\vec{x} \in \R^d} \D(\vec{x}; P) = \D(\vec{0}; P) = 1/2$.

Further, we examine the~properties of the~\textbf{HD} of $\alpha$-symmetric distributions in the~presence of contamination. Consider $P, Q\in\P{d}$ and $\varepsilon\in (0, 1)$. By $(1-\varepsilon)P+\varepsilon\, Q \in \P{d}$ we denote the~distribution $P$ which is $\varepsilon$-contaminated by $Q$, i.e.
\begin{equation*}
    ((1-\varepsilon)P+\varepsilon\, Q)(B)=(1-\varepsilon)P(B)+\varepsilon\, Q(B)\quad\text{for any Borel }B\subseteq\R^d.
\end{equation*}
These contaminated distributions form the~so-called Huber's contamination model. This model assumes that the~data may contain both `clean' observations from the~assumed distribution $P$ and `contaminating' observations from some other distribution $Q$ (outliers, faulty observations, etc.). Motivated by applications in machine learning and data science, recent years have seen a~growing interest in developing location and scatter estimators that maintain high accuracy even under contamination. Consider the~general problem of estimating a~parameter $\vec{\mu}=\vec{\mu}(P)$ with an~estimator $\widehat{\vec{\mu}}_n$ based on a~random sample $\set{\vec{X}_i}_{i=1}^n$ drawn from a~contaminated distribution $(1-\varepsilon)P+\varepsilon\, Q$. The~quality of this estimator can be assessed in various ways. Traditionally, statistical work has focused on expected risk measures, such as the~mean squared error $\E\norm{\widehat{\vec{\mu}}_n-\vec{\mu}}_2^2$. However, such risk measures can sometimes be misleading; specifically, when the~estimation deviation $\norm{\widehat{\vec{\mu}}_n-\vec{\mu}}_2$ lacks sufficient concentration, the~expected value may not accurately capture the~typical behavior of the~estimation deviation. Consequently, we seek estimators $\widehat{\vec{\mu}}_n$ that are `close' to $\vec{\mu}$ with high probability. Our objective, therefore, is to identify, for any given sample size $n$ and \emph{confidence parameter} $\delta \in (0, 1)$, the~smallest possible value $R(\delta, n, d, \varepsilon)$ for which
$\prob(\norm{\widehat{\vec{\mu}}_n-\vec{\mu}}_2 \leq R(\delta, n, d, \varepsilon)) \geq 1 - \delta$.

\subsection{Concentration inequality for maximal halfspace depth}    
\label{section: location rate of depth}

The following lemma provides the~concentration inequality for estimating the~maximum \textbf{HD} in Huber's contamination model. It does not require the~$\alpha$-symmetry of the~distribution $P$. The~proof of this lemma is a~direct modification of the~proof of~\citet[Theorem~2.1]{Chen_etal2018}, and for completeness, it is included in detail in the~Supplementary Material, Section~S.1.

\begin{lemma}\label{lemma: conv. of max depth}
    Let $P\in\P{d}$ be any distribution with a~halfspace median $\vec{\mu}\hs \in \R^d$ and let $\varepsilon<1/3$. Consider $\widehat{\vec{\mu}}_n\hs$ a~sample halfspace median based on a~random sample $\vec{X}_1, \ldots, \vec{X}_n$ drawn from a~contaminated distribution $(1-\varepsilon) P +\varepsilon\,Q$, where $Q\in\P{d}$. Then there exist absolute constants $C_1, C_2>0$ such that for any $\delta\in(0,1/2)$ the~inequality
    \begin{equation}\label{eq: DepthBound1}
        \prob\left(\abs{\D(\vec{\mu}\hs; P)-\D(\widehat{\vec{\mu}}_n\hs; P)}\leq \frac{\varepsilon}{1-\varepsilon}+C_1\sqrt{\frac{d}{n}}+C_2\sqrt{\frac{\log{(1/\delta)}}{n}}\right)\geq 1-2\delta
    \end{equation}
    holds for all $n\in \N$ such that
    \begin{equation*}
        \sqrt{\frac{\log(1/\delta)}{2\,n}}<1/3.
    \end{equation*}
\end{lemma}

The concentration inequality~\eqref{eq: DepthBound1} implies that 
\begin{equation*}
   \abs{\D(\vec{\mu}\hs; P) - \D(\widehat{\vec{\mu}}_n\hs; P)} \precsim \varepsilon + \sqrt{\frac{d}{n}}
\end{equation*} 
holds with high probability for a~large enough sample size $n$. Notably, Lemma~\ref{lemma: conv. of max depth} applies without any assumptions on $P$, meaning that no moment conditions are required. For any $t > 0$, we have
\begin{equation}\label{eq: decay}
    \prob\left(\abs{\D(\vec{\mu}\hs; P) - \D(\widehat{\vec{\mu}}_n\hs; P)} > \frac{\varepsilon}{1 - \varepsilon} + C_1\sqrt{\frac{d}{n}} + t \right) \leq 2 \exp\left( -\frac{n\, t^2}{C_2^2} \right).
\end{equation}
Thus, the~estimation deviation of maximum depth exhibits strong tail decay.

\subsection{Concentration inequality for halfspace median under \texorpdfstring{$\alpha$}{alpha}-symmetry} 

We now turn to the~task of estimating the~halfspace median $\vec{\mu}\hs$ of $\alpha$-symmetric random vector $\vec{X} \sim P\in \P{d}$ based on a~contaminated random sample $\set{\vec{X}_i}_{i=1}^n\sim (1-\varepsilon)P+\varepsilon\, Q$. We will assume that 
    \begin{enumerate}[label=($\mathsf{A}_2$), ref=\upshape{($\mathsf{A}_2$)}]
        \item \label{A2} the~distribution function $F$ of $X_1$ from~\eqref{eq2} satisfies the~condition
    \begin{equation}\label{eq5}
     \inf_{0<\abs{t}< \gamma}\frac{\abs{F(t)-F(0)}}{\abs{t}}\geq\kappa
    \end{equation}
    for some fixed constants $\gamma, \kappa >0$ such that $\varepsilon/(1-\varepsilon)<\gamma\kappa<1/2$.
    \end{enumerate} 
Recall that $F(0)=1/2$. Condition~\ref{A2} guarantees that the~marginal distribution function $F$ grows faster than $t\mapsto \kappa \, t+1/2$ on some appropriate neighborhood of $0$ depending on $\varepsilon$, the~amount of contamination. Note that~\eqref{eq5} also implies 
\begin{equation}\label{eq14}
    \quad\inf_{\abs{t}\geq\gamma}\abs{F(t)-F(0)}\geq\gamma\kappa,
\end{equation}
see Figure~\ref{figure: condition} for illustration. Finally, note that the~mapping $\varepsilon\mapsto \varepsilon/(1-\varepsilon)$ is strictly increasing on $(0, 1)$ and maps $1/3$ to $1/2$. Therefore, $\varepsilon/(1-\varepsilon)<\gamma\,\kappa<1/2$ from~\ref{A2} is never satisfied for $\varepsilon \geq 1/3$.

\begin{figure}[htpb]
    \centering
    \includegraphics[width=0.5\linewidth]{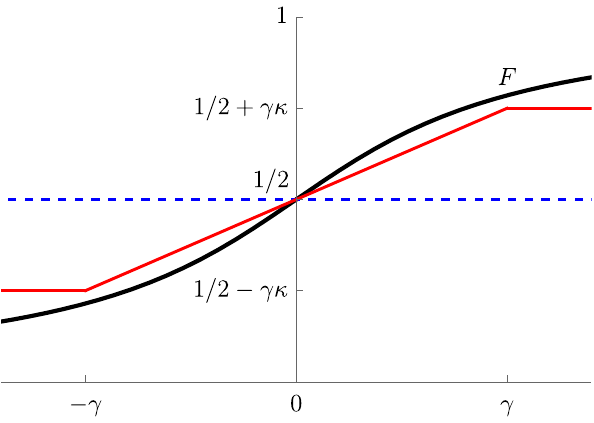}
    \caption{Condition~\ref{A2}: There must exist constants $\gamma, \kappa>0$ such that the~distribution function $F$ (black) does not extend into the~region between the~red line and the~blue dashed line. In other words, $F$ cannot be too `flat' around the~origin.}
    \label{figure: condition}
\end{figure}

\begin{note}
    For spherically symmetric distributions ($\alpha=2$), the~density of the~first marginal $f=F'$ always exists and is positive and continuous at~$0$ \citep[p.~37]{Fang_etal1990}. This implies that condition~\ref{A2} is always satisfied for some $\gamma, \kappa >0$. However, even for $\alpha = 2$, we will require~\ref{A2} to hold uniformly across all distributions in the~model.
\end{note}

We consider the~model of all $\alpha$-symmetric distributions, where $\alpha>0$ is fixed, and which are $\varepsilon$-contaminated by some other distribution. In Theorem~\ref{th: convFixed} below, we show that in such a~model, the~$\norm{\cdot}_\beta$-deviation (i.e., the~estimation deviation measured by the~$\beta$-norm $\norm{\cdot}_\beta$) of the~halfspace median $\widehat{\vec{\mu}}\hs_n$ can be bounded from above by $\left(\varepsilon+\sqrt{\frac{d}{n}}+\sqrt{\frac{\log{(1/\delta)}}{n}}\right)$ with probability at least $1-2\delta$. This result is consistent with the~finding of~\citet[Theorem~2.1]{Chen_etal2018} for $\alpha=\beta=2$, i.e., for spherically symmetric distributions.

\begin{theorem}\label{th: convFixed}
    Fix $\varepsilon \in(0, 1/3)$, $\alpha>0$ and let $\beta$ be the~conjugate index of $\alpha$ defined in equation~\eqref{eq: betaDef}. Then, for any  $\delta\in(0, 1/2)$, there exists an absolute constant $C>0$ such that the~inequality
    \begin{equation} \label{eq: ineqA}
        \inf_{P, Q} \prob\left(\norm{\widehat{\vec{\mu}}\hs_n-\vec{\mu}\hs}_\beta \leq C \left(\varepsilon+\sqrt{\frac{d}{n}}+\sqrt{\frac{\log{(1/\delta)}}{n}}\right)\right)\geq 1-2\delta
    \end{equation}
    holds for all $n\in\N$ such that 
    \begin{equation}\label{eq: condition on n}
        C_1\sqrt{\frac{d}{n}}+C_2\sqrt{\frac{\log{(1/\delta)}}{n}}<\gamma\kappa-\frac{\varepsilon}{1-\varepsilon},
    \end{equation}
    where $C_1, C_2>0$ are the~absolute constants from Lemma~\ref{lemma: conv. of max depth}. The~halfspace median $\widehat{\vec{\mu}}\hs_n$ in~\eqref{eq: ineqA} is based on a~random sample $\vec{X}_1, \ldots, \vec{X}_n\sim (1-\varepsilon)P+\varepsilon\,Q$. The~infimum in~\eqref{eq: ineqA} is taken over all $\alpha$-symmetric distributions $P$ such that condition~\ref{A2} holds uniformly with constants $\gamma, \kappa>0$ where $\varepsilon/(1-\varepsilon)<\gamma\kappa\leq1/2$, and over all contaminating distributions $Q\in\P{d}$.
\end{theorem}

\begin{note}
    Note that $\gamma\kappa-\varepsilon/(1-\varepsilon)>0$ by condition~\ref{A2}. Therefore, inequality~\eqref{eq: condition on n} is satisfied for all $n\geq n_0$, where $n_0$ depends on $\delta, d, \varepsilon, \gamma$ and $\kappa$. 
\end{note}

\begin{proof}[Proof of Theorem~\ref{th: convFixed}]
    Fix $\delta\in(0, 1/2)$ and consider $\gamma, \kappa>0$ such that condition~\ref{A2} is uniformly satisfied. Consider $n\in\N$ such that~\eqref{eq: condition on n} holds. Because $\gamma\kappa-\varepsilon/(1-\varepsilon)<1/2$ and $C_2>5$ (see equation~(S.8) in Supplementary Material, Section~S.1), this implies that 
    \begin{equation*}
        \sqrt{\frac{\log(1/\delta)}{2\,n}}<1/3.
    \end{equation*}
    By Lemma~\ref{lemma: conv. of max depth}, for any $P\in\P{d}$ and $\widehat{\vec{\mu}}_n\hs$ based on $\set{\vec{X}_i}_{i=1}^n\sim (1-\varepsilon)P+\varepsilon\,Q$, we have 
    \begin{equation}\label{eq10}
        \abs{\D(\vec{\mu}\hs; P)-\D(\widehat{\vec{\mu}}_n\hs; P)}\leq \frac{\varepsilon}{1-\varepsilon}+C_1\sqrt{\frac{d}{n}}+C_2\sqrt{\frac{\log{(1/\delta)}}{n}}
    \end{equation}
    with probability at least $1-2\delta$.
    
    Now, consider $\alpha$-symmetric distribution $P$ with the~distribution function of $X_1$ denoted by $F$. Note that $\vec{\mu}\hs=\vec{0}$. By formula~\eqref{eq3} for the~\textbf{HD} w.r.t. $P$ and bound \eqref{eq10}, we have with probability at least $1-2\delta$ that
    \begin{equation}\label{eq8}
    \abs{F\left(\norm{\widehat{\vec{\mu}}_n\hs\hspace{-2pt}-\hspace{-2pt}\vec{\mu}\hs}_\beta\right)\hspace{-2pt}-\hspace{-2pt}F(0)}=\abs{F\left(\norm{\widehat{\vec{\mu}}_n\hs}_\beta\right)\hspace{-2pt}-\hspace{-2pt}F(0)}  \leq\frac{\varepsilon}{1-\varepsilon}+C_1\sqrt{\frac{d}{n}}+C_2\sqrt{\frac{\log{(1/\delta)}}{n}}.
    \end{equation}
    By our choice of $n$~\eqref{eq: condition on n}, the~right-hand side of~\eqref{eq8} is strictly bounded from above by $\gamma\kappa$. Further, by~\eqref{eq14} we deduce
    \begin{equation*}
        \norm{\widehat{\vec{\mu}}_n\hs-\vec{\mu}\hs}_\beta< \gamma.
    \end{equation*}
    Note that $\varepsilon<1/3$, hence also $\varepsilon/(1-\varepsilon)<3/2\,\varepsilon$. In turn, condition~\ref{A2} together with~\eqref{eq8} implies that 
    \begin{align*}
        \norm{\widehat{\vec{\mu}}_n\hs-\vec{\mu}\hs}_\beta&\leq\frac{1}{\kappa}\abs{F\left(\norm{\widehat{\vec{\mu}}_n\hs-\vec{\mu}\hs}_\beta\right)-F(0)}  \leq \frac{1}{\kappa}\left(\frac{3}{2}\varepsilon+C_1\sqrt{\frac{d}{n}}+C_2\sqrt{\frac{\log{(1/\delta)}}{n}}\right)\\
        &\leq C \left(\varepsilon+\sqrt{\frac{d}{n}}+\sqrt{\frac{\log{(1/\delta)}}{n}}\right)
    \end{align*}
    holds with probability at least $1-2\delta$ for an appropriately chosen absolute constant $C>0$. This is true for any $\alpha$-symmetric $P$ satisfying condition~\ref{A2}, which concludes the~proof.
\end{proof}

For $\alpha=\beta=2$, we get the~upper bound of order $\left(\varepsilon+\sqrt{d/n}+\sqrt{\log{(1/\delta)}/n}\right)$ for estimation deviation measured by Euclidean norm. As shown by~\citet[Theorem~2.2]{Chen_etal2018}, this order is optimal (up to a~constant), indicating the~minimax optimality of the~halfspace median in the~model of contaminated spherically (and elliptically, see Remarks~\ref{remark: affine} and~\ref{remark: affine 2}) symmetric distributions. Thus, the~halfspace median achieves an optimal estimation deviation in such a~model. In the~model of $\alpha$-symmetric distributions, $\alpha\neq 2$, the~deviation still achieves an upper bound of order $\left(\varepsilon+\sqrt{d/n}+\sqrt{\log{(1/\delta)}/n}\right)$, but only if measured in the~$\beta$-norm. For example, consider $\alpha<2$, which implies $\beta>2$. We can use the~well-known inequality $\norm{\vec{x}}_2\leq d^{1/2-1/\beta} \norm{\vec{x}}_\beta$. Thus, in the~$\alpha$-symmetric distribution model, the~Euclidean deviation of the~halfspace median achieves an upper bound of order $d^{1/2-1/\beta}\left(\varepsilon+\sqrt{d/n}+\sqrt{\log{(1/\delta)}/n}\right)$.
The presence of the~factor $\sqrt{\log(1/\delta)}$ is particularly significant, as it reflects a~rapid decay in the~tail probability of estimation deviations. This is analogous to the~decay observed in the~concentration inequality for the~maximal halfspace depth~\eqref{eq: decay}, and suggests that the~halfspace median offers robust performance with high probability guarantees.

\begin{note}    \label{remark: affine 2}
In accordance with Remark~\ref{remark: affine}, one can also consider affine images of $\alpha$-symmetric distributions in Theorem~\ref{th: convFixed}. For $\beta\geq 1$, the~induced matrix $\beta$-norm of a~$d \times d$ matrix $\mat{A}$ is defined as
\begin{equation*}
    \norm{\mat{A}}_\beta = \sup_{\vec{x} \neq \vec{0}} \frac{\norm{\mat{A}\vec{x}}_\beta}{\norm{\vec{x}}_\beta}.
\end{equation*}
Let $M > 1$ be a~constant. Let $\vec{X} \sim P_\vec{X} \in \P{d}$ be $\alpha$-symmetric, and consider the~distribution of $\mat{A}\vec{X} + \vec{\mu} \sim P_{\mat{A}\vec{X} + \vec{\mu}} \in \P{d}$ for any $\vec{\mu} \in \R^d$ and non-singular $d \times d$ matrix $\mat{A}$ such that $\norm{\mat{A}}_\beta \leq M$. Denote by $\widehat{\vec{\mu}}\hs_n$ a~halfspace median based on a~random sample from $(1 - \varepsilon)P_{\mat{A}\vec{X} + \vec{\mu}} + \varepsilon\,Q$. Using the~affine equivariance of the~halfspace median from~\eqref{eq: affine equivariance}, we observe that $\widehat{\vec{\mu}}\hs_n$ estimates (ignoring the~contamination) the~halfspace median $\vec{\mu}\hs(P_{\mat{A}\vec{X} + \vec{\mu}}) = \vec{\mu}$ of $P_{\mat{A}\vec{X} + \vec{\mu}}$. Transforming this random sample via the~inverse affine mapping $\varphi\colon\vec{y} \mapsto \mat{A}^{-1}(\vec{y} - \vec{\mu})$, we obtain a~random sample from $(1 - \varepsilon)P_\vec{X} + \varepsilon\,Q'$ for $Q'\in\P{d}$ an affine transformation of $Q$. The~sample halfspace median based on the~latter transformed sample, denoted by $\widetilde{\vec{\mu}}\hs_n$, estimates (ignoring the~contamination again) the~halfspace median $\vec{\mu}\hs(P_\vec{X})=\vec{0}$ of $P_\vec{X}$. We can then bound
\begin{align*}
    \norm{\widehat{\vec{\mu}}\hs_n - \vec{\mu}\hs(P_{\mat{A}\vec{X} + \vec{\mu}})}_\beta &= \norm{\widehat{\vec{\mu}}\hs_n - \vec{\mu}}_\beta = \norm{\mat{A}\left(\widetilde{\vec{\mu}}\hs_n-\vec{\mu}\hs(P_\vec{X})\right)}_\beta \\ &\leq \norm{\mat{A}}_\beta \norm{\widetilde{\vec{\mu}}\hs_n-\vec{\mu}\hs(P_\vec{X})}_\beta \leq M \norm{\widetilde{\vec{\mu}}\hs_n-\vec{\mu}\hs(P_\vec{X})}_\beta.
\end{align*}
Therefore, for large enough $n$, we have as in Theorem~\ref{th: convFixed} 
\begin{equation*}
    \norm{\widehat{\vec{\mu}}\hs_n-\vec{\mu}\hs(P_{\mat{A}\vec{X} + \vec{\mu}})}_\beta \leq M C\left(\varepsilon+\sqrt{d/n}+\sqrt{\log{(1/\delta)}/n}\right)
\end{equation*}
uniformly with probability at least $1 - 2\delta$.
\end{note}

\begin{note}    
    Using the~inequality $\norm{\widehat{\vec{\mu}}\hs_n-\vec{\mu}\hs}_\beta\geq\norm{\widehat{\vec{\mu}}\hs_n-\vec{\mu}\hs}_\infty$, one can derive a~uniform result in Theorem~\ref{th: convFixed} across all $\alpha$-symmetric distributions by replacing the~$\beta$-norm with the~supremum norm $\norm{\cdot}_\infty$. Specifically, the~estimation error $\norm{\widehat{\vec{\mu}}\hs_n-\vec{\mu}\hs}_\infty$ can be bounded with high probability by $C \left(\varepsilon+\sqrt{\frac{d}{n}}+\sqrt{\frac{\log{(1/\delta)}}{n}}\right)$, uniformly over all $\alpha$-symmetric distributions $P$, $\alpha>0$, that satisfy condition~\ref{A2} with constants $\gamma, \kappa>0$ such that $\varepsilon/(1-\varepsilon)<\gamma\kappa\leq1/2$, and over all contaminating distributions $Q\in\P{d}$. This implies that within the~class of $\alpha$-symmetric distributions satisfying condition~\ref{A2}, the~upper bound $\left(\varepsilon+\sqrt{d/n}+\sqrt{\log{(1/\delta)}/n}\right)$ is attained in the~$\infty$-norm. 
\end{note}

\section{Scatter halfspace median of \texorpdfstring{$\alpha$}{alpha}-symmetric distributions}\label{sec:4}
In the~second part of this work, we are interested in estimating the~scatter parameter for $\alpha$-symmetric distributions $P \in \P{d}$ using the~\textbf{sHD}~\eqref{eq: SHD}. As far as we are aware, the~only available result on the~\textbf{sHD} of $\alpha$-symmetric distributions is the~expression for the~depth $\SHD(\mat{\Sigma}; P)$ given in \citet{Nagy2019}. That result, however, does not deal with the~scatter halfspace median matrix $\mat{\Sigma}\hs$ of $P$. This section provides several properties of $\mat{\Sigma}\hs$ that are of independent interest. We show (i) conditions under which the~\textbf{sHD} $\SHD(\mat{\Sigma}; P)$ and its maximizer $\mat{\Sigma}\hs(P)$ are continuous in the~argument of $P$, (ii) show that the~scatter halfspace median matrix is Fisher consistent under $\alpha$-symmetry of $P$, and (iii) give an explicit expression for this matrix, including a~proof of its uniqueness.

Recall that in the~definition~\eqref{eq: SHD} of the~\textbf{sHD}, we consider the~location functional $T$ to be the~halfspace median~\eqref{eq: Tukey median}. In all the~results of this section, it will be necessary that the~halfspace median $T$ is continuous at $P \in \P{d}$, meaning that whenever $P_n$ converges to $P$ weakly in $\P{d}$, then $T(P_n) \to T(P)$. This is true if $P$ is smooth, and $T(P)$ is unique \citep[Theorem~2]{Mizera_Volauf2002}. 

When discussing the~convergence of matrices, we always mean the~element-wise convergence of matrices (that is, convergence in the~Frobenius norm). Our first result expands~\citet[Theorem~3.1]{Paindaveine2018} that establishes (semi-)continuity of the~\textbf{sHD} in the~argument $\mat{\Sigma}$. We will need an analogous statement that is valid in both arguments $\mat{\Sigma}$ and $P$ of the~\textbf{sHD}.

\begin{theorem} \label{theorem: continuity}
The \textbf{sHD} mapping 
    \begin{equation*} 
    \SHD \colon \PD \times \P{d} \to [0,1] \colon \left( \mat{\Sigma}, P \right) \mapsto \SHD(\mat{\Sigma}; P) 
    \end{equation*}
is continuous in both arguments at any $(\mat{\Sigma}, P_0) \in \PD \times \P{d}$ such that $P_0$ is smooth, and $T(P_0)$ is unique. In other words, for any sequence of matrices $\left\{ \mat{\Sigma}_n \right\}_{n = 1}^\infty \subset \PD$ that converge to $\mat{\Sigma} \in \PD$, and for any sequence of distributions $\left\{ P_n \right\}_{n=1}^\infty \subset \P{d}$ that converge weakly to $P_0$ we have
    \[  {\lim}_{n \to \infty} \SHD(\mat{\Sigma}_n; P_n) = \SHD(\mat{\Sigma}; P_0). \]
\end{theorem}

\begin{proof}
First, we rewrite the~definition of the~\textbf{sHD}~\eqref{eq: SHD} in terms of slabs. For that, define the~slab centered at $\vec{\mu} \in \R^d$ in direction $\vec{u} \in \S{d}$ of width $2\,t \geq 0$ as
    \begin{equation} \label{eq: slab}  \Sl{\vec{\mu}, \vec{u}, t} = \left\{ \vec{x} \in \R^d \colon \abs{\inner{\vec{x} - \vec{\mu}, \vec{u}}} \leq t \right\}. \end{equation}
The closure of its complementary set is denoted by
    \begin{equation} \label{eq: cslab} \cSl{\vec{\mu}, \vec{u}, t} = \left\{ \vec{x} \in \R^d \colon \abs{\inner{\vec{x} - \vec{\mu}, \vec{u}}} \geq t \right\}. \end{equation}
This allows us to rewrite $\SHD(\mat{\Sigma}; P)$ with $\vec{X} \sim P$ into
    \[  
    \inf_{\vec{u} \in \S{d}} \min \left\{ \prob\left( \vec{X} \in \Sl{T(P), \vec{u}, \sqrt{\vec{u}\T \mat{\Sigma} \vec{u}}}\right), \prob\left( \vec{X} \in \cSl{T(P), \vec{u}, \sqrt{\vec{u}\T \mat{\Sigma} \vec{u}}}\right) \right\}.    
    \]
Thanks to our assumption of smoothness of $P_0$, all slabs~\eqref{eq: slab} and their complements~\eqref{eq: cslab} are continuity sets of $P_0$. Thus, one can apply the~portmanteau theorem \citep[Theorem~11.1.1]{Dudley2002} to see that both maps
    \begin{equation*} 
    \begin{aligned}
    \psi_{1,\vec{u}} & \colon \PD \times \P{d} \to [0,1] \colon (\mat{\Sigma}, P) \mapsto \prob\left( \vec{X} \in \Sl{T(P), \vec{u}, \sqrt{\vec{u}\T \mat{\Sigma} \vec{u}}}\right), \\
    \psi_{2,\vec{u}} & \colon \PD \times \P{d} \to [0,1] \colon (\mat{\Sigma}, P) \mapsto \prob\left( \vec{X} \in \cSl{T(P), \vec{u}, \sqrt{\vec{u}\T \mat{\Sigma} \vec{u}}}\right),
    \end{aligned}
    \end{equation*}
are continuous on their domain, at any $P \in \P{d}$ where $T$ is continuous. At $P_0$, the~latter continuity requirement on $T$ is verified by \citet[Theorem~2(iv)]{Mizera_Volauf2002}.
It remains to use Berge’s Maximum theorem \citep[pp. 115–117]{Berge1963} on parametric optimization to conclude that the~depth function $\SHD$, being an infimum of a~collection of continuous functions, is itself continuous at $(\mat{\Sigma}, P_0)$. 
%
\end{proof}


In the~following theorem, we give conditions under which the~scatter median set $\Smed(P)$ from~\eqref{eq: scatter median set} is continuous in the~argument of the~distribution $P \in \P{d}$.

\begin{theorem} \label{theorem: median continuity}
Let $P_0 \in \P{d}$ be smooth, and let both the~location halfspace median set $\Lmed(P_0)$ from~\eqref{eq: Tukey median} and the~scatter halfspace median set $\Smed(P_0)$ from~\eqref{eq: scatter median set} contain single elements. Denote the~location median of $P_0$ by $T(P_0) \in \R^d$ and the~scatter halfspace median matrix of $P_0$ by $\mat{\Sigma\hs} \in \PD$. Then, the~following holds true.
    \begin{enumerate}[label=(\roman*), ref=(\roman*)]
        \item For any sequence of distributions $\left\{ P_n \right\}_{n=1}^\infty \subset \P{d}$ converging weakly to $P_0$ and any sequence $\mat{\Sigma}_n \in \Smed(P_n)$ we have $\lim_{n \to \infty} \mat{\Sigma}_n = \mat{\Sigma}\hs$.
        \item Let $\widehat{P}_n \in \P{d}$ stand for the~empirical distribution of a~random sample $\vec{X}_1, \dots, \vec{X}_n$ from $P_0$. Then, the~sample scatter halfspace median matrix is strongly consistent, meaning that for any $\widehat{\mat{\Sigma}}_n \in \Smed(\widehat{P}_n)$ we have $\lim_{n\to\infty} \widehat{\mat{\Sigma}}_n = \mat{\Sigma}\hs$ almost surely.
    \end{enumerate} 
\end{theorem}

\begin{proof}
The first statement follows from the~continuity of the~\textbf{sHD} map established in Theorem~\ref{theorem: continuity} and Berge's Maximum theorem \citep[pp.~115--117]{Berge1963} applied directly to the~continuous map $\phi \colon \PD \times \mathcal P \to [0,1] \colon (\mat{\Sigma}, P) \mapsto \SHD(\mat{\Sigma}; P)$, where $\mathcal P = \{P_0\} \cup \left\{ P_n \colon n = 1, 2, \dots \right\}$. The~Maximum theorem implies that the~map $P \mapsto \Smed(P)$ from $\mathcal P$ to the~subsets of $\PD$ is an outer semi-continuous set-valued map in the~sense of \citet[Definition~5.4]{Rockafellar_Wets1998}. That means that for any $\mat{\Sigma}_n \in \Smed(P_n)$, all cluster points of the~sequence $\left\{ \mat{\Sigma}_n \right\}_{n=1}^\infty$ lie in $\Smed(P_0)$. The~same Maximum theorem gives that the~maximum depth map $\psi \colon \mathcal P \to [0,1] \colon P \mapsto \max_{\mat{\Sigma} \in \PD} \SHD(\mat{\Sigma}; P)$ is continuous. Thanks to a~tightness argument for the~convergent sequence $\left\{ P_n \right\}_{n=1}^\infty \subset \P{d}$ \citep[Theorem~11.5.4]{Dudley2002} and the~continuity of $\psi$, there must exist at least one cluster point of $\left\{ \mat{\Sigma}_n \right\}_{n=1}^\infty$ in $\PD$.\footnote{Observe that it is not possible that the~cluster point $\mat{\Sigma}_0$ of $\left\{ \mat{\Sigma}_n \right\}_{n=1}^\infty$ is a~singular matrix because in that case, a~straightforward modification of Theorem~\ref{theorem: continuity} gives that $\lim_{n\to\infty} \SHD(\mat{\Sigma}_n; P_n) = 0$, which contradicts the~outer semi-continuity of the~map $P \mapsto \Smed(P)$ established above.} Because $\Smed(P_0) = \set{\mat{\Sigma}\hs}$ is a~singleton, this means $\lim_{n\to\infty} \mat{\Sigma}_n = \mat{\Sigma}\hs$.

The second statement of the~theorem follows directly from the~first part and the~Varadarajan theorem \citep[Theorem~11.4.1]{Dudley2002} that establishes that $\widehat{P}_n$ converges weakly to $P_0$ as $n \to \infty$, almost surely.
\end{proof}

Having the~continuity of the~scatter halfspace median mapping and the~strong consistency of its sample version in Theorem~\ref{theorem: median continuity}, we now turn to the~specifics of the~scatter halfspace median for $\alpha$-symmetric distributions. 

The \textbf{sHD} for $\alpha$-symmetric distributions was treated already in \citet[Theorem~1 and formulas~(7) and~(8)]{Nagy2019}, where it was proved that for $\alpha$-symmetric distributions $P \in \P{d}$ we have
    \begin{equation} \label{eq: SHD for alpha}  
    \SHD(\mat{\Sigma}; P) = 2\, \min\left\{ F\left( \inf_{\vec{u} \in \S{d}} \frac{\sqrt{\vec{u}\T \mat{\Sigma} \vec{u}}}{\norm{\vec{u}}_{\alpha}} \right) - 1/2, 1 - F\left( \sup_{\vec{u} \in \S{d}} \frac{\sqrt{\vec{u}\T \mat{\Sigma} \vec{u}}}{\norm{\vec{u}}_{\alpha}} \right) \right\}.  
    \end{equation}
Here we used that for $\alpha$-symmetric distributions we know that $T(P) = \vec{0} \in \R^d$ and, again, write $F$ for the~distribution function of $X_1$, the~first marginal of $\vec{X} = (X_1, \dots, X_d)\T$.\footnote{Note that thanks to assumption~\ref{A1}, we do not need to use the~limit from the~left in the~second term in~\eqref{eq: SHD for alpha}, as was done in \citet[Theorem~1]{Nagy2019}.}

The next result gives the~explicit expression for the~scatter halfspace median matrix of any $\alpha$-symmetric distribution.

\begin{figure}[htpb]
    \centering
    \includegraphics[width=0.5\linewidth]{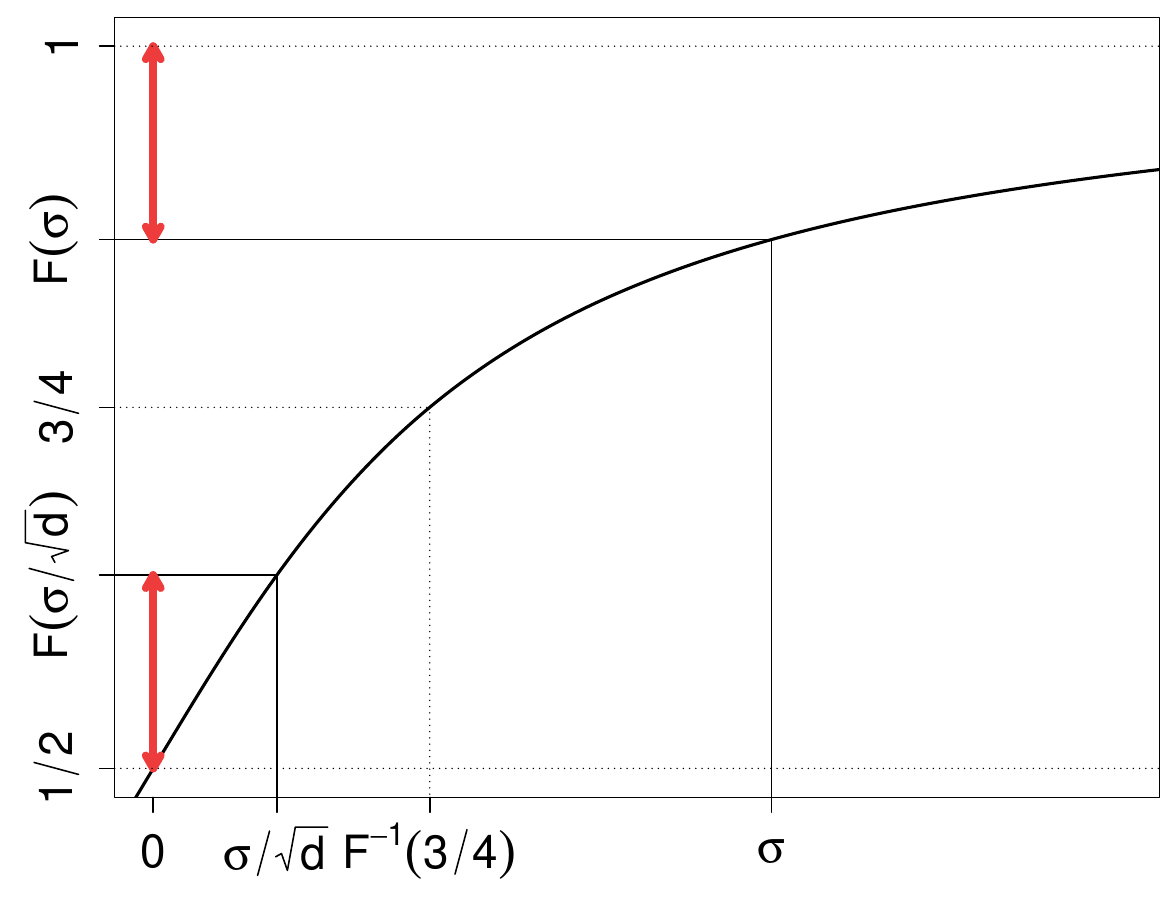}
    \caption{Distribution function $F$ of the~first marginal $X_1$ of an $\alpha$-symmetric distribution (in this case, $\alpha = 1$ and $F$ corresponds to the~Cauchy distribution). For $\alpha \ne 2$, no matrix $\mat{\Sigma} \in \PD$ can attain \textbf{sHD} $1/2$ \citep[Theorem~2]{Nagy2019}, which means that one of the~two expressions in the~minimum in~\eqref{eq: SHD for identity} must be smaller than $1/4$. The~two expressions $F(\sigma d^{1/2-1/\alpha}) - 1/2$ and $1 - F(\sigma)$ from~\eqref{eq: conditionSigma} are visualized in the~figure as the~lengths of the~red arrows (for $\alpha = 1$ and a~specific $\sigma > 0$). The~minimum of these two lengths is maximized if they are equal. Thus, the~maximum \textbf{sHD} is attained at $\sigma^2 \mat{I}$ such that~\eqref{eq: conditionSigma} is verified.}
    \label{figure: distribution function}
\end{figure}

\begin{theorem} \label{th: scatter median of symmetric}
For any $\alpha$-symmetric distribution $P \in \P{d}$, the~unique scatter halfspace median matrix is $\sigma^2 \mat{I}$, where $\sigma^2$ is the~unique solution of the~equation 
    \begin{equation}\label{eq: conditionSigma}
    F(\sigma\, d^{1/2 - 1/\alpha}) - 1/2 = 1 - F(\sigma).
    \end{equation} 
The maximum \textbf{sHD} of $P$ is 
    \begin{equation*}
    \SHD(\sigma^2 \mat{I}; P) = \max_{\mat{\Sigma} \in \PD} \SHD(\mat{\Sigma}; P) = 2F(\sigma\, d^{1/2 - 1/\alpha}) - 1=2 - 2F(\sigma).
    \end{equation*}
\end{theorem}

\begin{proof}
First, we show that the~\textbf{sHD} of $P$ is maximized at some multiple of the~identity matrix $\mat{I} \in \PD$. Write $P_{\mat{A}\vec{X}} \in \P{d}$ for the~distribution of the~random vector $\mat{A} \vec{X}$, for $\mat{A} \in \R^{d \times d}$ and $\vec{X} \sim P$, and denote by $\mat{\Sigma}\hs_{\mat{A}\vec{X}} \in \PD$ the~barycenter of the~scatter median set $\Smed(P_{\mat{A}\vec{X}})$. Thanks to the~affine equivariance of the~\textbf{sHD} \citep[Theorem~2.1]{Paindaveine2018}, we know that   
    \[
    \mat{\Sigma}\hs_{\mat{A}\vec{X}} = \mat{A} \mat{\Sigma}\hs_{\vec{X}} \mat{A}\T \quad \mbox{for all non-singular matrices $\mat{A} \in \R^{d \times d}$.}
    \]
Applying this result with $\mat{A}$ a~sign-permutation matrix as in Lemma~\ref{lemma: sign permutation}, we obtain the~identity
    \begin{equation}\label{eq: permMatrix}
    \mat{\Sigma}\hs_{\vec{X}} = \mat{A} \mat{\Sigma}\hs_{\vec{X}} \mat{A}\T \quad \mbox{for all sign-permutation matrices $\mat{A} \in \R^{d \times d}$.}    
    \end{equation}
The only matrices that satisfy~\eqref{eq: permMatrix} are multiples of the~identity matrix $\mat{I} \in \PD$. To see this, denote the~elements of $\mat{\Sigma}\hs_{\vec{X}}$ by $\sigma_{i,j}$, where $i,j\in\set{1,\ldots,d}$. Let $i\neq j$ and consider a~sign-permutation matrix $\mat{A}$ that swaps the~components $i,j$ and simultaneously reverses the~sign of component $j$ (that is, matrix $\mat{A}$ has the~only non-zero elements $a_{i,j} = -1$, $a_{j, i} = 1$, and $a_{\ell, \ell} = 1$ for all $\ell \ne i, j$). From~\eqref{eq: permMatrix} and the~symmetry of $\mat{\Sigma}\hs_{\vec{X}}$ we obtain $\sigma_{i,j}=-\sigma_{j,i}=-\sigma_{i,j}$, and  $\sigma_{i,i}=\sigma_{j,j}$. Necessarily, we get that $\mat{\Sigma}\hs_{\vec{X}} = \sigma^2 \mat{I}$ for some $\sigma > 0$.

To find the~specific value of $\sigma > 0$, we begin from the~expression for the~\textbf{sHD}~\eqref{eq: SHD for alpha}, which we want to maximize over all matrices $\mat{\Sigma} = \sigma^2 \mat{I}$. We obtain
    \begin{equation} \label{eq: SHD for identity}  
    \begin{aligned}
    \SHD(\sigma^2 \mat{I}; P) & = 2\, \min\left\{ F\left( \sigma \inf_{\vec{u} \in \S{d}} \frac{\sqrt{\vec{u}\T \vec{u}}}{\norm{\vec{u}}_{\alpha}} \right) - 1/2, 1 - F\left( \sigma \sup_{\vec{u} \in \S{d}} \frac{\sqrt{\vec{u}\T \vec{u}}}{\norm{\vec{u}}_{\alpha}} \right) \right\} \\
    & = 2\, \min\left\{ F\left( \sigma \inf_{\vec{u} \in \S{d}} \frac{\norm{\vec{u}}_2}{\norm{\vec{u}}_{\alpha}} \right) - 1/2, 1 - F\left( \sigma \sup_{\vec{u} \in \S{d}} \frac{\norm{\vec{u}}_2}{\norm{\vec{u}}_{\alpha}} \right) \right\}, \\
    \end{aligned}
    \end{equation}
see also Figure~\ref{figure: distribution function}. Further argumentation is carried out separately for the~three considered cases (i) $\alpha < 2$, (ii) $\alpha > 2$, and (iii) $\alpha = 2$. The~following lemma will be useful.

\begin{lemma}\label{lemma: zeroMatrix}
Let $\mat{A}=\left(a_{i,j}\right)_{i,j=1}^d$ be a~symmetric matrix such that $a_{j,j}=0$ for all $j = 1, \dots, d$ and $\vec{v}\T\mat{A}\vec{v}\geq 0$ for all $\vec{v}\in\set{-1,1}^d$. Then $\mat{A}$ is the~zero matrix $\mat{A} = \mat{0} \in \R^{d \times d}$.
\end{lemma}

\begin{proof}[Proof of Lemma~\ref{lemma: zeroMatrix}]
We give a~proof by induction on $d$. For $d=1$ the~lemma holds trivially. Consider $d>1$ and let the~assertion hold for $d-1$. Choose arbitrary $\vec{v}=\left(v_1, \ldots, v_d\right)\T\in\set{-1,1}^d$ and let $\widetilde{\vec{v}}=\left(v_1,\ldots, v_{d-1}, -v_d\right)\T$. By the~assumption (recall that $a_{d,d}=0$), we have that
    \begin{equation}\label{eq: matLemma1}
        \vec{v}\T\mat{A}\vec{v}=\sum_{i,j=1}^{d-1}v_i v_j a_{i,j}+2v_d\sum_{i=1}^{d-1}v_i a_{i,d}\geq 0
    \end{equation}
and
    \begin{equation}\label{eq: matLemma2}
        \widetilde{\vec{v}}\T\mat{A}\widetilde{\vec{v}}=\sum_{i,j=1}^{d-1}v_i v_j a_{i,j}-2v_d\sum_{i=1}^{d-1}v_i a_{i,d}\geq 0.
    \end{equation}
By summing equations~\eqref{eq: matLemma1} and~\eqref{eq: matLemma2} we obtain that
    \begin{equation*}
    0 \leq \sum_{i,j=1}^{d-1}v_i v_j a_{i,j} =     \hat{\vec{v}}\T\hat{\mat{A}}\hat{\vec{v}}
    \end{equation*}
holds for any $\hat{\vec{v}}=\left(v_1, \ldots, v_{d-1}\right)\T\in\set{-1,1}^{d-1}$, where $\hat{\mat{A}} =\left(a_{i,j}\right)_{i,j=1}^{d-1}$ is obtained from matrix $\mat{A}$ by removing the~last column and row. This means that $\hat{\mat{A}}$ satisfies the~assumptions of this lemma and by the~induction hypothesis, $\hat{\mat{A}}$ is the~zero matrix  $\hat{\mat{A}} = \mat{0} \in \R^{(d-1)\times(d-1)}$, i.e. $a_{i,j}=0$ for all $i, j = 1,\dots,d-1$. By plugging this matrix back into~\eqref{eq: matLemma1} and~\eqref{eq: matLemma2}, we obtain that
    \begin{equation}    \label{eq: last condition on v}
    v_d\sum_{i=1}^{d-1}v_i a_{i,d}=0 \quad \mbox{for all }\vec{v} = (v_1, \dots, v_d)\T \in \{-1, 1\}^d.
    \end{equation}
To conclude, let $j \in \set{1,\dots, d-1}$ and consider vectors $\vec{u} =(1,\ldots,1)\T \in \{ -1, 1 \}^{d}$ and $\vec{w} = \vec{u} - 2\vec{e}_j \in \{-1, 1\}^{d}$. Then, taking these vectors in~\eqref{eq: last condition on v} gives 
    \begin{equation*}
    0=u_d\sum_{i=1}^{d-1}u_i a_{i,d}-w_d\sum_{i=1}^{d-1}w_i a_{i,d}=\sum_{i=1}^{d-1}a_{i,d}-\sum_{i=1}^{d-1}a_{i,d}+2a_{j,d}=2a_{j,d},
    \end{equation*}
hence $a_{j,d}=0$ for all $j\in\set{1,\ldots,d}$.
\end{proof}

We can now proceed with the~proof of Theorem~\ref{th: scatter median of symmetric}.

\proofpart{(i)}{Case $\alpha < 2$}
Using the~standard inequalities for $\alpha$-norms with $\alpha < 2$
    \begin{equation}    \label{eq: Lp inequalities}
    \begin{aligned}
    \norm{\vec{u}}_2 & \leq \norm{\vec{u}}_\alpha \leq d^{1/\alpha-1/2} \norm{\vec{u}}_2 \quad \mbox{for all $\vec{u} \in \R^d$,}    
    \end{aligned}
    \end{equation}
we see that~\eqref{eq: SHD for identity} simplifies to
    \begin{equation*}
    \begin{aligned}
    \SHD(\sigma^2 \mat{I}; P) & = 2\, \min\left\{ F\left( \sigma \, d^{1/2 - 1/\alpha} \right) - 1/2, 1 - F\left( \sigma \right) \right\}.
    \end{aligned}
    \end{equation*}
To maximize the~last expression, the~constant $\sigma$ must chosen so that the~\textbf{sHD} value, as visualized using arrows in Figure~\ref{figure: distribution function}, is as large as possible. That naturally gives
    \[
    F(\sigma\, d^{1/2 - 1/\alpha}) - 1/2 = 1 - F(\sigma) = \SHD(\sigma^2 \mat{I}; P)/2.  
    \]
The unique solution to this equation provides the~specific expression for $\sigma$ for $\alpha < 2$. 

It remains to show that the~scatter halfspace median of $P$ is unique. Suppose that there is a~matrix $\mat{\Sigma} \in \PD$ such that $\SHD(\mat{\Sigma}; P) = \SHD(\sigma^2 \mat{I}; P)$. The~left-hand side inequality in~\eqref{eq: Lp inequalities} turns into equality if $\vec{u} = \vec{e}_i$ is one of the~canonical basis vectors of $\R^d$, $i = 1, \dots, d$. We know that the~support of any $\alpha$-symmetric measure $P$ with $\alpha \ne 2$ is $\R^d$ \citep[Theorem~2]{Misiewicz1992}. This means that for $\mat{\Sigma}$ to attain the~same \textbf{sHD} as $\sigma^2 \mat{I}$, in each direction $\vec{v} \in \S{d}$ where
    \begin{equation*} 
    1 = \sup_{\vec{u} \in \S{d}} \frac{\norm{\vec{u}}_2}{\norm{\vec{u}}_\alpha} = \frac{\norm{\vec{v}}_2}{\norm{\vec{v}}_\alpha} 
    \end{equation*}
it must be true that 
    \[  
    \frac{\sqrt{\vec{v}\T \mat{\Sigma} \vec{v}}}{\norm{\vec{v}}_\alpha} \leq \sigma \quad \mbox{for each }\vec{v} = \vec{e}_i\mbox{ for }i = 1, \dots, d. 
    \]
Otherwise, because $F$ is strictly increasing on $\R$, by~\eqref{eq: SHD for alpha} we would get 
    \[
    \SHD(\mat{\Sigma}; P) \leq 2 \left( 1 - F\left( \frac{\sqrt{\vec{v}\T \mat{\Sigma} \vec{v}}}{\norm{\vec{v}}_\alpha} \right) \right) < 2 \left( 1 - F\left( \sigma \right) \right) = \SHD(\sigma^2 \mat{I}; P).
    \]
Altogether, we obtain that for $\mat{\Sigma}$ to attain the~maximum \textbf{sHD}, necessarily
    \begin{equation} \label{eq: condition1}  
    \sqrt{\vec{e}_i\T \mat{\Sigma} \vec{e}_i} \leq \sigma \norm{\vec{e}_i}_\alpha = \sigma \quad \mbox{for each }i = 1, \dots, d. 
    \end{equation}
An analogous set of constraints can be imposed on $\mat{\Sigma}$ also based on the~right-hand side inequality in~\eqref{eq: Lp inequalities}. There, equality is attained if and only if $\vec{u}$ is a~positive multiple of the~vector 
    \begin{equation} \label{eq: v} \vec{v} = \left(\pm 1, \pm 1, \dots, \pm 1\right)\T \in \R^d,  
    \end{equation}
where by $\pm 1$ we mean that any element of this vector may be $1$ with either a~positive or a~negative sign, and these signs may differ from one element to another. At each such vector $\vec{v}$, it must be true for $\mat{\Sigma}$ that
    \[  
    \frac{\sqrt{\vec{v}\T \mat{\Sigma} \vec{v}}}{\norm{\vec{v}}_\alpha} \geq \sigma \, d^{1/2-1/\alpha} \quad \mbox{for each }\vec{v} = \left(\pm 1, \pm 1, \dots, \pm 1\right)\T \in \R^d,
    \]
for otherwise, analogously as before, \eqref{eq: SHD for alpha} would imply
    \[
    \SHD(\mat{\Sigma}; P) \leq 2 \left( F\left( \frac{\sqrt{\vec{v}\T \mat{\Sigma} \vec{v}}}{\norm{\vec{v}}_\alpha} \right) - 1/2 \right) < 2 \left( F\left( \sigma\, d^{1/2-1/\alpha} \right) - 1/2 \right) = \SHD(\sigma^2 \mat{I}; P),
    \]
which goes against our assumption that $\mat{\Sigma}$ maximizes the~\textbf{sHD}. We obtain that $\mat{\Sigma}$ must obey the~constraints
    \begin{equation} \label{eq: condition2}  
    \sqrt{\vec{v}\T \mat{\Sigma} \vec{v}} \geq \sigma \, d^{1/2-1/\alpha} \norm{\vec{v}}_\alpha = \sigma \, \sqrt{d} \quad \mbox{for each }\vec{v} = \left( \pm 1, \dots, \pm 1\right)\T \in \R^d. 
    \end{equation}
To finalize our proof, it remains to show that the~two sets of conditions~\eqref{eq: condition1} and~\eqref{eq: condition2} already imply that $\mat{\Sigma} = \sigma^2 \mat{I}$.

Denote the~elements of the~symmetric positive definite matrix $\mat{\Sigma} \in \PD$ by $\sigma_{i,j}$, $i, j = 1, \dots, d$. Condition~\eqref{eq: condition1} then gives that $\sigma_{i,i} \leq \sigma^2$ for each $i = 1, \dots, d$. For each $\vec{v} = (\pm 1, \dots, \pm 1)\T$, condition~\eqref{eq: condition2} gives $\vec{v}\T \mat{\Sigma} \vec{v} \geq d\, \sigma^2$. Because all the~elements of $\vec{v}$ are either $1$ or $-1$, one can express $\vec{v}\T \mat{\Sigma} \vec{v}$ as
    \[  \vec{v}\T \mat{\Sigma} \vec{v} = \vec{v}\T (\mat{\Sigma}_0 + \diag(\mat{\Sigma})) \vec{v} = \vec{v}\T \mat{\Sigma}_0 \vec{v} + \trace(\mat{\Sigma}), \]
where $\mat{\Sigma}_0 = \mat{\Sigma} - \diag(\mat{\Sigma})$, $\diag(\mat{\Sigma}) \in \R^{d \times d}$ is the~diagonal matrix with the~same entries on its main diagonal as $\mat{\Sigma}$, and $\trace(\mat{\Sigma}) = \sum_{i=1}^d \sigma_{i, i}$ is the~trace of $\mat{\Sigma}$. Combining the~two inequalities in~\eqref{eq: condition1} and~\eqref{eq: condition2}, we obtain that for each $\vec{v} = (\pm 1, \dots, \pm 1)\T$ it necessarily must be true that
    \[
    d\, \sigma^2 \leq \vec{v}\T \mat{\Sigma} \vec{v} = \vec{v}\T \mat{\Sigma}_0 \vec{v} + \trace(\mat{\Sigma}) \leq \vec{v}\T \mat{\Sigma}_0 \vec{v} + d\, \sigma^2.
    \]
Thus, the~matrix $\mat{\Sigma}_0$ with zero diagonal must obey
    \[  0 \leq \vec{v}\T \mat{\Sigma}_0 \vec{v} \quad \mbox{for all }\vec{v} = (\pm 1, \dots, \pm 1)\T. \]
By Lemma~\ref{lemma: zeroMatrix}, this condition is equivalent with the~matrix $\mat{\Sigma}_0$ being the~zero matrix $\mat{0} \in \R^{d \times d}$. Consequently, for~\eqref{eq: condition1} and~\eqref{eq: condition2} to be both satisfied, $\mat{\Sigma}$ must be diagonal. The~two conditions~\eqref{eq: condition1} and~\eqref{eq: condition2} then directly imply $\mat{\Sigma} = \sigma^2 \, \mat{I}$, and we have verified that the~only matrix maximizing $\SHD(\cdot; P)$ must be $\sigma^2 \, \mat{I}$. The~proof for $\alpha < 2$ is concluded.

\proofpart{(ii)}{Case $\alpha > 2$}
For $\alpha > 2$, we proceed in complete analogy with the~case $\alpha < 2$. The~inequalities between $\alpha$-norms with $\alpha > 2$ now take the~form
    \begin{equation}   \label{eq: Lp inequalities 2}
    \begin{aligned}
    1 \leq \frac{\norm{\vec{u}}_2}{\norm{\vec{u}}_\alpha} \leq d^{1/2 - 1/\alpha} \quad \mbox{for all $\vec{u} \in \R^d$,}
    \end{aligned}
    \end{equation}
which gives that the~\textbf{sHD}~\eqref{eq: SHD for identity} simplifies to
    \[
    \SHD(\sigma^2 \mat{I}; P) = 2\, \min\left\{ F\left( \sigma \right) - 1/2, 1 - F\left( \sigma \, d^{1/2 - 1/\alpha} \right) \right\},
    \]
which is again maximized if~\eqref{eq: conditionSigma} is true. For $\vec{u} = \vec{e}_i$, $i = 1,\dots, d$, we attain equality on the~left-hand side of~\eqref{eq: Lp inequalities 2}, while for $\vec{u}$ a~positive multiple of $\vec{v}$ from~\eqref{eq: v} we get equality on the~right-hand side of~\eqref{eq: Lp inequalities 2}. Plugging these vectors into the~general expression for the~\textbf{sHD}~\eqref{eq: SHD for alpha}, we obtain a~set of inequalities for $\mat{\Sigma} \in \PD$ maximizing $\SHD(\cdot; P)$ analogous to those in the~case $\alpha < 2$, and Lemma~\ref{lemma: zeroMatrix} again concludes that necessarily $\mat{\Sigma} = \sigma^2 \mat{I}$.

\proofpart{(iii)}{Case $\alpha = 2$}
In the~spherically symmetric case $\alpha = 2$ we get $\norm{\vec{u}}_2/\norm{\vec{u}}_\alpha = 1$ for all $\vec{u} \in \S{d}$ in~\eqref{eq: SHD for identity}. This gives
    \[
    \SHD(\sigma^2\mat{I}; P) = 2 \min\left\{ F(\sigma) - 1/2, 1 - F(\sigma) \right\},
    \]
which is clearly maximized if $F(\sigma) = 3/4$. This is also a~special case of the~formula~\eqref{eq: conditionSigma}. Using the~fact that $F$ is strictly increasing at $\sigma$ \citep[Theorem~2.10]{Fang_etal1990}, we get that the~only maximizer $\mat{\Sigma} \in \PD$ of the~depth $\SHD(\cdot; P)$ must satisfy $\sqrt{\vec{u}\T \mat{\Sigma} \vec{u}} = \sigma$ for all $\vec{u} \in \S{d}$, which is obviously true only for $\mat{\Sigma} = \sigma^2 \mat{I}$.
\end{proof}

\section{Estimation of scatter halfspace median under contamination}
\label{sec:5}
In this section, we address the~problem of determining an upper bound for estimating the~scatter halfspace median matrix under $\alpha$-symmetry. In Section~\ref{section: rate for elliptical}, we present a~concentration inequality for the~\textbf{sHD} of the~scatter halfspace median matrix following from the~results of~\citet{Chen_etal2018} and recover the~upper bound for estimating the~scatter parameter of spherical distributions~\citep[Theorem~3.1]{Chen_etal2018}. However, this method cannot be directly extended to the~case where $\alpha \neq 2$. To address this limitation, in Section~\ref{section: alphaSHD}, we introduce a~modification of the~\textbf{sHD} that is well-suited for the~statistical analysis of $\alpha$-symmetric distributions.

\subsection{Concentration inequality for scatter halfspace median in spherical setting}
\label{section: rate for elliptical}

Similarly as for the~location halfspace median in Section~\ref{section: location rate of depth}, the~first step for establishing the~upper bound for the~scatter median matrix under contamination is to find the~rate for its \textbf{sHD}. This is done in the~following lemma.

\begin{lemma}   \label{lemma: SHD rate}
    Let $P\in\P{d}$ be any distribution such that its scatter halfspace median matrix $\mat{\Sigma}\hs$ exists, and let $\varepsilon<1/3$. Consider $\widehat{\mat{\Sigma}}_n\hs$ a~sample scatter halfspace median matrix based on a~random sample $\vec{X}_1, \ldots, \vec{X}_n$ drawn from a~contaminated distribution $(1-\varepsilon) P +\varepsilon\,Q$, where $Q\in\P{d}$. Then there exist absolute constants $C_1, C_2>0$ such that for any $\delta\in(0,1/2)$ the~inequality
    \begin{equation}\label{eq: DepthBound2}
        \prob\left(\abs{\SHD(\mat{\Sigma}\hs; P)-\SHD(\widehat{\mat{\Sigma}}_n\hs; P)}\leq\frac{\varepsilon}{1-\varepsilon}+C_1\sqrt{\frac{d}{n}}+C_2\sqrt{\frac{\log{(1/\delta)}}{n}}\right)\geq 1-2\delta
    \end{equation}
    holds for all $n\in \N$ such that
    \begin{equation*}
        \sqrt{\frac{\log(1/\delta)}{2\,n}}<1/3.
    \end{equation*}
\end{lemma}
\begin{proof}
    This proof is entirely analogous to that of Lemma~\ref{lemma: conv. of max depth}. It closely follows the~approach of~\citet[Theorem 7.1]{Chen_etal2018}, with the~same minor modification as in the~proof of Lemma~\ref{lemma: conv. of max depth} applied.
\end{proof}

Same as for the~location \textbf{HD}, the~concentration inequality~\eqref{eq: DepthBound2} implies that 
\begin{equation*}
   \abs{\SHD(\mat{\Sigma}\hs; P)-\SHD(\widehat{\mat{\Sigma}}_n\hs; P)} \precsim \varepsilon + \sqrt{\frac{d}{n}}
\end{equation*} 
holds with high probability for a~large enough sample size $n$. Lemma~\ref{lemma: SHD rate} applies without any assumptions on $P$ other than its scatter halfspace median matrix must exist. Same as before, for any $t > 0$, we have
\begin{equation*}
    \prob\left(\abs{\SHD(\mat{\Sigma}\hs; P)-\SHD(\widehat{\mat{\Sigma}}_n\hs; P)} > \frac{\varepsilon}{1 - \varepsilon} + C_1\sqrt{\frac{d}{n}} + t \right) \leq 2 \exp\left( -\frac{n\, t^2}{C_2^2} \right),
\end{equation*}
indicating strong tail decay.

For $\alpha = 2$ and $P \in \P{d}$ spherically symmetric, the~unique scatter halfspace median is the~matrix $\mat{\Sigma}\hs = \sigma^2 \mat{I}$ with $\sigma = F^{-1}(3/4)$, thanks to~\eqref{eq: conditionSigma}. Let $\widehat{\mat{\Sigma}}_n\hs$ be a~sample scatter half\-space median matrix based on a~random sample $\vec{X}_1, \ldots, \vec{X}_n$ drawn from a~contaminated distribution $(1-\varepsilon) P +\varepsilon\,Q$, where $Q\in\P{d}$. Using Lemma~\ref{lemma: SHD rate} and~\eqref{eq: SHD for alpha} we get that for large $n$ with probability at least $1-2\delta$ we have
    \[  
    \begin{aligned} 
    & \frac{\varepsilon}{1-\varepsilon} +C_1\sqrt{\frac{d}{n}}+C_2\sqrt{\frac{\log{(1/\delta)}}{n}} \geq \abs{\SHD(\mat{\Sigma}\hs; P)-\SHD(\widehat{\mat{\Sigma}}_n\hs; P)} \\
    & = 2 \abs{\frac{1}{4} - \min\left\{ F\left( \inf_{\vec{u} \in \S{d}} \sqrt{\vec{u}\T \widehat{\mat{\Sigma}}_n\hs \vec{u}} \right) - 1/2, 1 - F\left( \sup_{\vec{u} \in \S{d}} \sqrt{\vec{u}\T \widehat{\mat{\Sigma}}_n\hs \vec{u}} \right) \right\}}, 
    \end{aligned}
    \]
which is equivalent with
    \begin{equation}    \label{eq: bound scatter 3/4}
    \sup_{\vec{u} \in \S{d}} \abs{F\left(\sqrt{\vec{u}\T \widehat{\mat{\Sigma}}_n\hs \vec{u}} \right) -  \frac{3}{4}} \leq \frac{\varepsilon}{2(1-\varepsilon)}+\frac{C_1}{2}\sqrt{\frac{d}{n}}+\frac{C_2}{2}\sqrt{\frac{\log{(1/\delta)}}{n}}.
    \end{equation}
Assume now a~condition on the~growth of $F$ that is analogous to~\ref{A2} from the~location case.
     \begin{enumerate}[label=($\mathsf{A}_3$), ref=\upshape{($\mathsf{A}_3$)}]
        \item \label{A3} The~marginal distribution function $F$ satisfies the~condition
        \begin{equation*}
        \inf_{0<\abs{t-\sigma^2}< \gamma}\frac{\abs{F\left(\sqrt{t}\right)-F\left(\sqrt{\sigma^2}\right)}}{\abs{t - \sigma^2}}\geq\kappa
        \end{equation*}
    for some fixed constants $\gamma, \kappa >0$ such that $\varepsilon/(2(1-\varepsilon))<\gamma\kappa\leq1/4$. 
    \end{enumerate} 
This is equivalent to the~first part of the~condition from~\citet[formula~(11)]{Chen_etal2018}. Condition~\ref{A3} implies that 
\begin{equation}\label{eq: scSphericalBound}
    \inf_{\abs{t-\sigma^2}\geq \gamma}\abs{F\left(\sqrt{t}\right)-F\left(\sqrt{\sigma^2}\right)}=\inf_{\abs{t-\sigma^2}\geq \gamma}\abs{F\left(\sqrt{t}\right)-\frac{3}{4}}\geq\gamma\kappa,
\end{equation}
therefore we need the~restriction $\gamma\kappa\leq1/4$. Because $\varepsilon/(2(1-\varepsilon))<\gamma\kappa$, consider $n$ large enough so that
\begin{equation*}
    \frac{\varepsilon}{2(1-\varepsilon)}+\frac{C_1}{2}\sqrt{\frac{d}{n}}+\frac{C_2}{2}\sqrt{\frac{\log{(1/\delta)}}{n}}<\gamma\kappa.
\end{equation*}
Formula~\eqref{eq: scSphericalBound} together with $\eqref{eq: bound scatter 3/4}$ then gives that $\abs{\vec{u}\T \widehat{\mat{\Sigma}}_n\hs \vec{u}-\sigma^2}<\gamma$ for all $\vec{u}\in\S{d}$. As a~consequence, condition~\ref{A3} and formula~\eqref{eq: bound scatter 3/4} imply that for $n$ large
\begin{align}
    & \norm{\widehat{\mat{\Sigma}}_n\hs - \mat{\Sigma}\hs}_{\mathrm{op}} =\sup_{\vec{u} \in \S{d}} \abs{\vec{u}\T \widehat{\mat{\Sigma}}_n\hs \vec{u} - \vec{u}\T \mat{\Sigma}\hs \vec{u}}=\sup_{\vec{u} \in \S{d}} \abs{\vec{u}\T \widehat{\mat{\Sigma}}_n\hs \vec{u} - \sigma^2}\nonumber \\ &\leq \frac{1}{2\,\kappa}\left(\frac{\varepsilon}{1-\varepsilon}+C_1\sqrt{\frac{d}{n}}+C_2\sqrt{\frac{\log{(1/\delta)}}{n}}\right) \leq C\left(\varepsilon+\sqrt{\frac{d}{n}}+\sqrt{\frac{\log{(1/\delta)}}{n}}\right)\label{eqSpherIneq}
\end{align}
for $C>0$. This holds for all $n\in\N$ such that
\begin{equation*}
    C_1\sqrt{\frac{d}{n}}+C_2\sqrt{\frac{\log{(1/\delta)}}{n}}<2\gamma\kappa-\frac{\varepsilon}{1-\varepsilon}.
\end{equation*}
In particular, we are able to recover the~minimax optimal rate of convergence for the~scatter halfspace median matrix as in \citet[Theorem~4.1]{Chen_etal2018}.

\sloppy In contrast to spherically symmetric distributions, where the~concentration inequality for the~\textbf{sHD}~of the~scatter halfspace median matrix $\mat{\Sigma}\hs = \sigma^2 \mat{I}$ (as given in Lemma~\ref{lemma: SHD rate}) suffices to derive an upper bound for the~sample scatter halfspace median $\widehat{\mat{\Sigma}}\hs_n$, this does not hold for general $\alpha$-symmetric distributions $P \in \P{d}$. The~challenge arises from the~quadratic form $\vec{u} \mapsto \sqrt{\vec{u}\T \mat{\Sigma} \vec{u}} = \norm{\mat{\Sigma}^{1/2} \vec{u}}_2$, which is compatible only with the~$2$-norm. This limitation is discussed in greater detail in Supplementary Material, Section~S.2. Nevertheless, in the~following section, we propose an alternative method for estimating the~scatter parameter of $\alpha$-symmetric distributions with $\alpha \neq 2$. This approach enables us to establish a~similar upper bound.

\subsection{Scatter halfspace depth adjusted for \texorpdfstring{$\alpha$}{alpha}-symmetric distributions}  \label{section: alphaSHD}

The incompatibility of the~\textbf{sHD} with the~$\alpha$-norm can be resolved by introducing an adjusted scatter halfspace depth, specifically suited for $\alpha$-symmetric distributions. For $\alpha > 0$ given, we introduce the~\emph{$\alpha$-scatter halfspace depth} (abbreviated as $\alpha$-\textbf{sHD}) of $\mat{\Sigma} \in \PD$ w.r.t. $P \in \P{d}$ as
    \begin{equation}\label{eq: alphaSHD}
    \begin{aligned}  
    \SHD_\alpha (\mat{\Sigma}; P) 
    & = \inf_{\vec{u} \in \S{d}} \min 
    \left\{ 
    \prob\left( \abs{\inner{\vec{X}-T(P), \vec{u}}} \leq \norm{\mat{\Sigma}^{1/2} \vec{u}}_\alpha \right),\right. \\
    & \hspace{8em}
    \left. \prob\left( \abs{\inner{\vec{X}-T(P), \vec{u}}} \geq \norm{\mat{\Sigma}^{1/2} \vec{u}}_\alpha \right) 
    \right\}, 
    \end{aligned}
    \end{equation}
where $\mat{\Sigma}^{1/2} \in \PD$ is the~unique positive definite square root matrix of $\mat{\Sigma}$ \citep[Theorem~7.2.6]{Horn_Johnson2013} that satisfies $\mat{\Sigma}^{1/2}\mat{\Sigma}^{1/2} = \mat{\Sigma}$. Of course, $T(P)$ in~\eqref{eq: alphaSHD} is the~halfspace median of $P$, $\SHD_2$ is the~standard \textbf{sHD}~\eqref{eq: SHD}, and the~empirical $\alpha$-\textbf{sHD} is $\SHD_\alpha(\cdot; \widehat{P}_n)$ for $\widehat{P}_n \in\P{d}$ the~empirical distribution of $P$. The~following theorem establishes the~basic properties of the~$\alpha$-\textbf{sHD}. 

\begin{theorem}\label{th: alphaSHD properties}
The $\alpha$-\textbf{sHD}~\eqref{eq: alphaSHD} has the~following properties.
    \begin{enumerate}[label=(\roman*), ref=(\roman*)]
        \item\label{alphaDepth: continuity} The~$\alpha$-\textbf{sHD} mapping 
        \begin{equation*} 
            \SHD_\alpha \colon \PD \times \P{d} \to [0,1] \colon \left( \mat{\Sigma}, P \right) \mapsto \SHD_\alpha(\mat{\Sigma}; P) 
        \end{equation*}
        is continuous in both arguments at any $(\mat{\Sigma}, P) \in \PD \times \P{d}$ such that $P$ is smooth and the~halfspace median $T(P)$ is unique.

        \item\label{alphaDepth: consistency} Let $P \in \P{d}$ be smooth, and suppose that both the~location \textbf{HD} and the~$\alpha$-\textbf{sHD} with respect to $P$ are uniquely maximized at $T(P)\in\R^d$ and $\mat{\Sigma}\hs_\alpha \in \PD$, respectively. Then, the~following holds:
        \begin{enumerate}[label=(\alph*), ref=(\alph*)]
        \item For any sequence $\left\{ P_n \right\}_{n=1}^\infty \subset \P{d}$ converging weakly to $P$ and any sequence $\mat{\Sigma}_n$ of maximizers of $\SHD_\alpha(\cdot; P_n)$ it holds that $\lim_{n \to \infty} \mat{\Sigma}_n = \mat{\Sigma}\hs_\alpha$.
            \item Denote the~empirical distribution of a~random sample $\vec{X}_1, \dots, \vec{X}_n$ from $P$ by $\widehat{P}_n \in \P{d}$. Then, the~sample $\alpha$-scatter halfspace median matrix is strongly consistent, meaning that for any sequence $\widehat{\mat{\Sigma}}_n$ of maximizers of $\SHD_\alpha(\cdot; \widehat{P}_n)$, we have $\lim_{n\to\infty} \widehat{\mat{\Sigma}}_n = \mat{\Sigma}\hs_\alpha$ almost surely.
        \end{enumerate} 

        \item\label{alphaDepth: permInv} The~$\alpha$-\textbf{sHD} is equivariant under signed permutation transformations. That is, for any $P_{\vec{X}} \in \P{d}$, $\mat{\Sigma} \in \PD$, and any signed permutation matrix $\mat{A}$, we have
        \[
        \SHD_\alpha(\mat{A}\mat{\Sigma}\mat{A}^\top; P_{\mat{A}\vec{X}}) = \SHD_\alpha(\mat{\Sigma}; P_{\vec{X}}).
        \]

        \item\label{alphaDepth: convRate} Let $P\in\P{d}$ be a~distribution such that its $\alpha$-scatter median matrix $\mat{\Sigma}_\alpha\hs$ exists, and let $\varepsilon<1/3$. Consider a~sequence of sample $\alpha$-scatter median matrices $\widehat{\mat{\Sigma}}_n$ based on a~random sample $\vec{X}_1, \ldots, \vec{X}_n$ drawn from a~contaminated distribution $(1-\varepsilon) P +\varepsilon\,Q$, where $Q\in\P{d}$. Then there exist absolute constants $C_1, C_2>0$ such that for any $\delta\in(0,1/2)$ the~inequality
        \begin{equation*}
            \left|\SHD_\alpha(\mat{\Sigma}_\alpha\hs; P)-\SHD_\alpha(\widehat{\mat{\Sigma}}_n; P)\right|\leq\frac{\varepsilon}{1-\varepsilon}+C_1\sqrt{\frac{d}{n}}+C_2\sqrt{\frac{\log{(1/\delta)}}{n}}
        \end{equation*}
        holds with probability at least $1 - 2 \delta$ for all $n\in \N$ such that
        \begin{equation*}
            \sqrt{\frac{\log(1/\delta)}{2\,n}}<1/3.
        \end{equation*}

    \end{enumerate}
\end{theorem}

\begin{proof}
    Note that the~definition~\eqref{eq: alphaSHD} of $\alpha$-\textbf{sHD} can be rewritten as
    \begin{align*}
        \SHD_\alpha (\mat{\Sigma}; P) = \inf_{\vec{u} \in \S{d}} \min&\left\{\prob\left( \vec{X} \in \Sl{T(P), \vec{u}, \norm{\mat{\Sigma}^{1/2} \vec{u}}_\alpha}\right),\right.\\ &\;\,\left.\prob\left( \vec{X} \in \cSl{T(P), \vec{u}, \norm{\mat{\Sigma}^{1/2} \vec{u}}_\alpha}\right)\right\},
    \end{align*}
    where $\vec{X} \sim P$. Assertions~\ref{alphaDepth: continuity} and~\ref{alphaDepth: consistency} follow by directly adapting the~proofs of Theorem~\ref{theorem: continuity} and Theorem~\ref{theorem: median continuity}. Part~\ref{alphaDepth: permInv} follows from the~definition~\eqref{eq: alphaSHD}. Assertion~\ref{alphaDepth: convRate} is established similarly to Lemma~\ref{lemma: SHD rate}; the~proof is a~modification of the~argument of~\citet[Theorem 7.1]{Chen_etal2018}. The~only difference in the~reasoning is that we consider slabs $\Sl{T(P), \vec{u}, \norm{\mat{\Sigma}^{1/2} \vec{u}}_\alpha}$ of width $2\norm{\mat{\Sigma}^{1/2} \vec{u}}_\alpha$ instead of $2\sqrt{\vec{u}\T \mat{\Sigma} \vec{u}}$.
\end{proof}

Employing the~projection property~\eqref{eq1} of the~$\alpha$-symmetric distributions, the~same derivation as in \citet[Theorem~1]{Nagy2019} gives that for $P \in \P{d}$ that is $\alpha$-symmetric with the~first marginal distribution function $F$, the~expression~\eqref{eq: SHD for alpha} changes to
    \begin{equation}\label{eq: alphaSHD for alpha}  
    \begin{aligned}
    \SHD_\alpha(\mat{\Sigma}; P) & = 2\, \min\left\{ F\left( \inf_{\vec{u} \in \S{d}} \frac{\norm{\mat{\Sigma}^{1/2} \vec{u}}_\alpha}{\norm{\vec{u}}_{\alpha}} \right) - 1/2, \right. \\
    & \hspace{8em} \left. 1 - F\left( \sup_{\vec{u} \in \S{d}} \frac{\norm{\mat{\Sigma}^{1/2} \vec{u}}_\alpha}{\norm{\vec{u}}_{\alpha}} \right) \right\}. 
    \end{aligned}
    \end{equation}
Unlike the~standard \textbf{sHD} of $P$, the~$\alpha$-\textbf{sHD}~\eqref{eq: alphaSHD for alpha} can attain the~maximum possible value of $1/2$. The~following theorem identifies the~associated $\alpha$-scatter halfspace median matrix.

\begin{theorem}
Let $P \in \P{d}$ be $\alpha$-symmetric with the~first marginal distribution function $F$ from~\eqref{eq2}. Then,
    \begin{enumerate}[label=(\roman*), ref=(\roman*)]
        \item \label{thm13 i} the~$\alpha$-scatter halfspace depth $\SHD_\alpha(\cdot; P)$ is uniquely maximized at $\mat{\Sigma}\hs_\alpha = \sigma^2 \mat{I}\in\PD$, where $\sigma = F^{-1}(3/4)$, with the~maximum $\alpha$-\textbf{sHD} of $1/2$. 

        \item \label{thm13 ii} Assume that
            \begin{enumerate}[label=\upshape{($\mathsf{A}_4$)}, ref=\upshape{($\mathsf{A}_4$)}]
            \item \label{A4} the~marginal distribution function $F$ of the~$\alpha$-symmetric distribution $P$ satisfies the~condition
            \begin{equation*}
            \inf_{0<\abs{t-\sigma}< \gamma}\frac{\abs{F\left(t\right)-F\left(\sigma\right)}}{\abs{t - \sigma}}\geq\kappa
            \end{equation*}
            for some fixed constants $\gamma, \kappa >0$ such that $\varepsilon/(2(1-\varepsilon))<\gamma\kappa\leq1/4$. 
            \end{enumerate} 
        Denote by $\widehat{\mat{\Sigma}}_n \in \PD$ the~$\alpha$-scatter halfspace median based on a~random sample $\vec{X}_1, \ldots, \vec{X}_n\sim (1-\varepsilon)P+\varepsilon\,Q$. Then, for any $\delta\in(0, 1/2)$, there exists an absolute constant $C>0$ such that
        \begin{align}    \label{eq: rate alpha}
            \sup_{\vec{u} \in \S{d}} \abs{\frac{\norm{\widehat{\mat{\Sigma}}_n^{1/2}\vec{u}}_\alpha}{\norm{\vec{u}}_\alpha} - \frac{\norm{\left(\mat{\Sigma}_\alpha\hs\right)^{1/2}\vec{u}}_\alpha}{\norm{\vec{u}}_\alpha}} &= \sup_{\vec{u} \in \S{d}} \abs{\frac{\norm{\widehat{\mat{\Sigma}}_n^{1/2}\vec{u}}_\alpha}{\norm{\vec{u}}_\alpha} - \sigma}\\ &\leq C\left(\varepsilon+\sqrt{\frac{d}{n}}+\sqrt{\frac{\log{(1/\delta)}}{n}}\right)\nonumber
        \end{align} 
        holds with probability at least $1-2\delta$ for all sufficiently large $n$. This holds uniformly over all $\alpha$-symmetric distributions $P \in \P{d}$ such that condition~\ref{A4} is uniformly satisfied, and over all contaminating distributions $Q\in\P{d}$.
    \end{enumerate}
\end{theorem}

\begin{proof}
For $\alpha=2$, part~\ref{thm13 i} follows directly from Theorem~\ref{th: scatter median of symmetric}. Now, consider the~case $\alpha \neq 2$. By assumption~\ref{A1}, we have that $P$ is smooth, so the~$\alpha$-\textbf{sHD} of any matrix is bounded from above by $1/2$. Let $\sigma = F^{-1}(3/4)$. Using~\eqref{eq: alphaSHD for alpha} we obtain
    \begin{equation*}
        \SHD_\alpha(\sigma^2\mat{I}; P) = 2\, \min\left\{ F\left(\sigma \right) - 1/2, 1 - F\left( \sigma \right) \right\}=1/2.
    \end{equation*}
Consider any matrix $\mat{\Sigma}\in\PD$ such that $\SHD_\alpha(\mat{\Sigma}; P)=1/2$. From~\eqref{eq: alphaSHD for alpha}, we deduce
    \begin{equation*}
        \inf_{\vec{u} \in \S{d}} \frac{\norm{\mat{\Sigma}^{1/2} \vec{u}}_\alpha}{\norm{\vec{u}}_{\alpha}}=\sup_{\vec{u} \in \S{d}} \frac{\norm{\mat{\Sigma}^{1/2} \vec{u}}_\alpha}{\norm{\vec{u}}_{\alpha}} =\sigma,
    \end{equation*}
which implies that 
    \begin{equation*}
        \norm{\frac{1}{\sigma}\mat{\Sigma}^{1/2} \vec{u}}_\alpha = \norm{\vec{u}}_\alpha\quad \text{for all }\vec{u}\in\R^d.
    \end{equation*}
This gives that the~function $f\colon \vec{u}\mapsto\mat{\Sigma}^{1/2}\vec{u}/\sigma$ maps the~unit sphere with respect to the~$\alpha$-norm onto itself. By~\citet[Corollary~2.4]{An2005} (for $\alpha<1$) and \citet{LiSo1994} (for $\alpha\geq 1, \alpha\neq 2$), it follows that $\mat{\Sigma}^{1/2}/\sigma$ is a~signed permutation matrix. In particular, $(\mat{\Sigma}^{1/2}/\sigma)^{-1}=(\mat{\Sigma}^{1/2}/\sigma)\T=\mat{\Sigma}^{1/2}/\sigma$, hence $\mat{\Sigma}^{1/2}=\sigma^2\mat{\Sigma}^{-1/2}$. Here, we used the~fact that the~inverse of any (signed) permutation matrix is equal to its transpose and that $\mat{\Sigma}\in\PD$. This implies that $\mat{\Sigma}=\sigma^2\mat{I}$. Consequently, $\mat{\Sigma}\hs_\alpha = \sigma^2 \mat{I}$ is the~unique deepest matrix with $\alpha$-\textbf{sHD} of $1/2$, and we have shown part~\ref{thm13 i}.
    
For part~\ref{thm13 ii}, we can apply the~same reasoning as in Section~\ref{section: rate for elliptical}. Specifically, using part~\ref{alphaDepth: convRate} of Theorem~\ref{th: alphaSHD properties} and the~form of the~$\alpha$-\textbf{sHD} for $\alpha$-symmetric distributions~\eqref{eq: alphaSHD for alpha}, we obtain that for sufficiently large $n$
    \begin{equation*}  
        \sup_{\vec{u} \in \S{d}} \abs{F\left(\frac{\norm{\widehat{\mat{\Sigma}}_n^{1/2}\vec{u}}_\alpha}{\norm{\vec{u}}_\alpha}\right) -  \frac{3}{4}} \leq  \frac{\varepsilon}{2(1-\varepsilon)}+\frac{C_1}{2}\sqrt{\frac{d}{n}}+\frac{C_2}{2}\sqrt{\frac{\log{(1/\delta)}}{n}}
    \end{equation*}
holds. This, combined with condition~\ref{A4}, implies that 
    \begin{align*}
        \sup_{\vec{u} \in \S{d}}& \abs{\frac{\norm{\widehat{\mat{\Sigma}}_n^{1/2}\vec{u}}_\alpha}{\norm{\vec{u}}_\alpha} - \sigma} = \sup_{\vec{u} \in \S{d}} \abs{\frac{\norm{\widehat{\mat{\Sigma}}_n^{1/2}\vec{u}}_\alpha}{\norm{\vec{u}}_\alpha} - \frac{\norm{\left(\mat{\Sigma}_\alpha\hs\right)^{1/2}\vec{u}}_\alpha}{\norm{\vec{u}}_\alpha}} \\ &\leq \frac{1}{2\,\kappa}\left(\frac{\varepsilon}{1-\varepsilon}+C_1\sqrt{\frac{d}{n}}+C_2\sqrt{\frac{\log{(1/\delta)}}{n}}\right) \\
        & = C\left(\varepsilon+\sqrt{\frac{d}{n}}+\sqrt{\frac{\log{(1/\delta)}}{n}}\right)
    \end{align*}
holds for a~sufficiently large sample size $n$, which concludes the~proof.
\end{proof}

The expression on the~left-hand side of~\eqref{eq: rate alpha} can be interpreted as a~distance between $\widehat{\mat{\Sigma}}_n$ and $\mat{\Sigma}\hs_\alpha$ with respect to the~pseudometric~\citep[p.~26]{Dudley2002} on the~space $\PD$ defined by
\begin{equation}\label{eq: pseudometric}
\begin{aligned}
    \dist(\mat{A}, \mat{B}) & = \sup_{\vec{u} \in \S{d}} \abs{ \frac{\norm{\mat{A}^{1/2}\vec{u}}_\alpha}{\norm{\vec{u}}_\alpha} - \frac{\norm{\mat{B}^{1/2}\vec{u}}_\alpha}{\norm{\vec{u}}_\alpha}
    } \\
    & = \sup_{\vec{u} \colon \norm{\vec{u}}_\alpha = 1} \abs{ \norm{\mat{A}^{1/2}\vec{u}}_\alpha - \norm{\mat{B}^{1/2}\vec{u}}_\alpha
    }.
\end{aligned}
\end{equation}
Thus, we have shown that the~$\alpha$-scatter halfspace median achieves an upper bound of order $\varepsilon+\sqrt{d/n}+\sqrt{\log{(1/\delta)}/n}$ with respect to the~pseudometric~\eqref{eq: pseudometric} when estimating the~scatter parameter of $\alpha$-symmetric distribution under Huber's contamination model. This result is analogous to the~original upper bound in \eqref{eqSpherIneq}, which holds for the~standard scatter halfspace depth and spherical distributions.

\section*{Acknowledgements}

\section*{Funding}
This work was supported by the~Czech Science Foundation (project no. 24-10822S), by the~ERC CZ grant LL2407 of the~Ministry of Education, Youth and Sport of the~Czech Republic, and by the~Grant Agency of Charles University (project no. 60125).

\section*{Author contributions}
CRediT: \textbf{Filip Bo\v{c}inec}: Conceptualization, Formal analysis, Investigation, Writing – original draft. \textbf{Stanislav Nagy}: Conceptualization, Formal analysis, Investigation, Supervision, Writing – review \& editing, Funding acquisition.

\section*{Disclosure statement}
The authors report there are no competing interests to declare.

\section*{ORCID}
\textbf{Filip Bo\v{c}inec}: \href{https://orcid.org/0009-0000-4415-9302}{https://orcid.org/0009-0000-4415-9302}\\
\textbf{Stanislav Nagy}: \href{https://orcid.org/0000-0002-8610-4227}{https://orcid.org/0000-0002-8610-4227}

\section*{Online supplementary material}
A pdf document: Proof of Lemma~\ref{lemma: conv. of max depth} and additional technical details.


\begingroup
\small
\setlength{\bibsep}{2pt}
\bibliographystyle{apalike}
\bibliography{bibliography}
\endgroup

\end{document}


\maketitle

\renewcommand{\thesection} {S.\arabic{section}} 
\renewcommand{\theequation}{S.\arabic{equation}}
\renewcommand{\thefigure}{S.\arabic{figure}}
\renewcommand{\thetable}{S.\arabic{table}}
\renewcommand{\thetheorem}{S.\arabic{theorem}}
\renewcommand{\theexample}{S.\arabic{example}}

\section{Proof of Lemma~\ref{lemma: conv. of max depth}}

\newtheorem{auxlemma}{Lemma}
\renewcommand*{\theauxlemma}{A\arabic{auxlemma}}

This proof follows the~steps of the~proof by~\citet[Theorem~2.1]{Chen_etal2018}. Only minor modifications have been made in order to include cases when $1/5\leq \varepsilon <1/3$. Throughout the~proof, the~sample \textbf{HD} of $\vec{x}\in\R^d$ w.r.t. a~random sample points $\vec{X}_1, \dots, \vec{X}_n$ with empirical distribution $\widehat{P}_n \in \P{d}$ is also denoted by $\D(\vec{x}; \set{\vec{X}_i}_{i=1}^n)=\D(\vec{x}; \widehat{P}_n)$. We begin by stating two auxiliary lemmata.

\begin{auxlemma}\label{App1}
    Let $P\in \P{d}$ and consider the~empirical distribution $\widehat{P}_n$ based on a~random sample of size $n$ from $P$. Then for all $\delta\in(0, 1)$ the~inequality
    \begin{equation*}
        \sup_{H \in \mathcal{H}_d}\abs{P(H)-\widehat{P}_n(H)}\leq \sqrt{\frac{1440 \pi e}{1-e^{-1}}}\sqrt{\frac{d+1}{n}}+\sqrt{\frac{\log{(1/\delta)}}{2n}}
    \end{equation*}
    holds with probability at~least $1-\delta$, where by $\mathcal{H}_d$ we denote the~system of all closed halfspaces in $\R^d$, i.e. all sets in the~form $\set{\vec{x} \in\R^d \colon\inner{\vec{x}, \vec{u}}\geq t}$ for $\vec{u}\in\S{d}$ and $t\in\R$.
\end{auxlemma}
\begin{proof}
This can~be proven in the~same way as \citep[Lemma~7.3]{Chen_etal2018} using that~the~VC dimension of $\mathcal{H}_d$ is $d+1$.
\end{proof}

\begin{auxlemma}\label{App3}
    Let $N\sim\mathsf{Binomial}(n, p)$ and assume $p<1/3$. Then, for every $\delta\in(0, 1)$ satisfying $\sqrt{\frac{\log(1/\delta)}{2\,n}}<1/3$, we have
    \begin{equation*}
        \frac{N}{n-N}\leq \frac{p}{1-p}+\frac{9}{2}\sqrt{\frac{\log(1/\delta)}{2\,n}}<2
    \end{equation*}
    with probability at~least $1-\delta$.
\end{auxlemma}
\begin{proof}
    The~proof is a~slight modification of~\citep[Lemma~7.1]{Chen_etal2018}. By Hoeffding's inequality~\citep[Section~2.1.2]{Wainwright2019} we have $\prob(N>np+t)\leq\exp(-2t^2/n)$ for all $t>0$. Set $t=\sqrt{n\log(1/\delta)/2}$ so that~with probability at~least $1-\delta$ we have $N\leq np+\sqrt{n\log(1/\delta)/2}$, hence also $n-N\geq n(1-p)-\sqrt{n\log(1/\delta)/2}$. As a~result
    \begin{equation}\label{eq: appIneq}
        \frac{N}{n-N}\leq \frac{p+\sqrt{\log(1/\delta)/(2n)}}{(1-p)-\sqrt{\log(1/\delta)/(2n)}}
    \end{equation}  
    holds with probability at~least $1-\delta$. Note that~for any $a,b\in (0, 1/3)$ we have that~$(a+b)/(1-a-b)\leq a/(1-a)+9b/2$. To see this, multiply this inequality with a~positive quantity $(1-a-b)(1-a)$ to obtain an~equivalent inequality $1\leq 9(1-a)(1-a-b)/2$, which is obviously true since $a,b\in(0, 1/3)$. Also, $a/(1-a)+9b/2<2$. Setting $a=p$ and $b=\sqrt{\log(1/\delta)/(2n)}$ in~\eqref{eq: appIneq} concludes the~proof.
\end{proof}

The~proof of Lemma~\ref{lemma: conv. of max depth} is divided into two parts.     
    
    \proofpart{1}{Auxiliary observations}
    First, we prepare the~following observations that~will be useful in deriving the~intended bounds.
    \begin{enumerate}[label=($\mathsf{L}_{\arabic*}$), ref=($\mathsf{L}_{\arabic*}$)]
        \item \label{L1} Consider a~random sample $\set{\vec{X}_i}_{i=1}^n\sim(1-\varepsilon) P +\varepsilon\,Q$. We can~decompose $\set{\vec{X}_i}_{i=1}^n=\set{\vec{X}_i, i\in N_1}\cup \set{\vec{X}_i, i\in N_2}$ where $N_1\cup N_2 = \set{1,\ldots, n}$, $N_1, N_2$ are disjoint, $\set{\vec{X}_i, i\in N_1}$ is a~random sample from $P$ and $\set{\vec{X}_i, i\in N_2}$ is a~random sample from $Q$. Denote by $n_1$ and $n_2$ the~cardinalities of $N_1$ and $N_2$, respectively. Note that~$n_2$ and $n_1=n-n_2$ are random variables such that~$n_2\sim\mathsf{Binomial}(n, \varepsilon)$ holds marginally.

        \item \label{L2} By Lemma~\ref{App1}, we have with probability at~least $1-\delta$ that
        \begin{multline*}
            \qquad\sup_{\vec{x}\in\R^d}\abs{\D(\vec{x}; P)-\D(\vec{x}; \set{\vec{X}_i, i\in N_1})} \\ \leq \sup_{H \in \mathcal{H}_d}\abs{P(H)-\widehat{P}_{n_1}(H)}\leq \sqrt{\frac{1440 \pi e}{1-e^{-1}}}\sqrt{\frac{d+1}{n_1}}+\sqrt{\frac{\log{(1/\delta)}}{2n_1}},
        \end{multline*}
        where $\widehat{P}_{n_1}$ is the~empirical distribution of $\set{\vec{X}_i, i\in N_1}$ and $\mathcal{H}_d$ is the~system of all closed halfspaces in $\R^d$.

        \item \label{L3} By the~definition of the~sample \textbf{HD}, it follows that
        \begin{equation*}
            \quad n_1\,\D(\vec{x}; \set{\vec{X}_i, i\in N_1}) \geq n\,\D(\vec{x}; \set{\vec{X}_i}_{i=1}^{n})-n_2 \geq n_1\,\D(\vec{x}; \set{\vec{X}_i, i\in N_1})-n_2
        \end{equation*}
        for all $\vec{x}\in\R^d$. For example, to see the~first inequality, note that~
        \begin{equation}\label{eq43}
            \inf_{\vec{u}\in\S{d}}\sum_{i\in N_1}\indic_{\set{\inner{\vec{X}_i, \vec{u}}\geq\inner{\vec{x}, \vec{u}}}}\geq \inf_{\vec{u}\in\S{d}}\sum_{i=1}^{n}\indic_{\set{\inner{\vec{X}_i, \vec{u}}\geq\inner{\vec{x}, \vec{u}}}}-n_2.
        \end{equation}
        This is because, for a~fixed $\vec{u}\in\S{d}$, the~left-hand side of \eqref{eq43} is the~number of observations from $\set{\vec{X}_i, i\in N_1}$ in the~halfspace $H_{\vec{x}, \vec{u}} \hspace{-2pt}=\hspace{-2pt} \set{\vec{y} \in \R^d \colon \inner{\vec{y},\vec{u}}\geq\inner{\vec{x},\vec{u}}}$, which is always greater than~or equal to the~number of observations from $\set{\vec{X}_i}_{i=1}^n$ in $H_{\vec{x}, \vec{u}}$ without $n_2$. That~is because some of the~observations from $\set{\vec{X}_i, i\in N_2}$ can~also lie in $H_{\vec{x}, \vec{u}}$. The~second inequality is proven analogously.

        \item By Lemma~\ref{App3}, if
        \begin{equation}\label{eq4}
            \sqrt{\frac{\log(1/\delta)}{2\,n}}<1/3
        \end{equation}
        holds for $\delta\in(0,1/2)$, then
        \begin{equation}\label{eq9}
            \prob\left[\frac{n_2}{n_1}\leq\frac{\varepsilon}{1-\varepsilon}+\frac{9}{2}\sqrt{\frac{\log(1/\delta)}{2n}}<2\right]\geq 1-\delta.
        \end{equation}
        Further, note that
        \begin{equation}\label{eq12}
            \frac{n_2}{n_1}<2\iff n_2< 2n_1\iff n-n_1< 2n_1\iff n_1> n/3.
        \end{equation}
        This means that~the~random event in~\eqref{eq9} implies that~at~least $1/3$ of all observations are non-contaminating.
    \end{enumerate}

    \proofpart{2}{The~intended bound}
    Let $n\in\N$ such that~\eqref{eq4} is satisfied. By~\ref{L1}, decompose $\set{\vec{X}_i}_{i=1}^n=\set{\vec{X}_i, i\in N_1}\cup\set{\vec{X}_i, i\in N_2}$. We derive the~following series of inequalities. These hold with probability at~least $1-\delta$ conditionally on the~decomposition $N_1, N_2$. We have
    \begin{equation} \label{eq11a}
    \begin{aligned}
         \D(\widehat{\vec{\mu}}\hs_n;P)&\stackrel{\text{\ref{L2}}}{\geq} \D(\widehat{\vec{\mu}}\hs_n; \set{\vec{X}_i, i\in N_1})-\sqrt{\frac{1440 \pi e}{1-e^{-1}}}\sqrt{\frac{d+1}{n_1}}-\sqrt{\frac{\log{(1/\delta)}}{2n_1}} \\
        &\stackrel{\text{\ref{L3}}}{\geq}\frac{n}{n_1}\D(\widehat{\vec{\mu}}\hs_n; \set{\vec{X}_i}_{i=1}^{n})-\frac{n_2}{n_1}-\sqrt{\frac{1440 \pi e}{1-e^{-1}}}\sqrt{\frac{d+1}{n_1}}-\sqrt{\frac{\log{(1/\delta)}}{2n_1}} \\
        &\;\geq\frac{n}{n_1}\D(\vec{\mu}\hs; \set{\vec{X}_i}_{i=1}^{n})-\frac{n_2}{n_1}-\sqrt{\frac{1440 \pi e}{1-e^{-1}}}\sqrt{\frac{d+1}{n_1}}-\sqrt{\frac{\log{(1/\delta)}}{2n_1}} \\
        &\stackrel{\text{\ref{L3}}}{\geq} \D(\vec{\mu}\hs; \set{\vec{X}_i, i\in N_1})-\frac{n_2}{n_1}-\sqrt{\frac{1440 \pi e}{1-e^{-1}}}\sqrt{\frac{d+1}{n_1}}-\sqrt{\frac{\log{(1/\delta)}}{2n_1}} \\
        &\stackrel{\text{\ref{L2}}}{\geq} \D(\vec{\mu}\hs; P)-\frac{n_2}{n_1}-2\sqrt{\frac{1440 \pi e}{1-e^{-1}}}\sqrt{\frac{d+1}{n_1}}-\sqrt{\frac{2\log{(1/\delta)}}{n_1}}
    \end{aligned}
    \end{equation}
    where the~third inequality follows from the~fact that~$\widehat{\vec{\mu}}\hs_n$ is the~maximizer of $\D(\cdot; \set{\vec{X}_i}_{i=1}^{n})$. Rewriting~\eqref{eq11a}, we have that
    \begin{equation*}
        \prob\left[\abs{\D(\vec{\mu}\hs; P)-\D(\widehat{\vec{\mu}}_n\hs; P)}\leq \frac{n_2}{n_1}+2\sqrt{\frac{1440 \pi e}{1-e^{-1}}}\sqrt{\frac{d+1}{n_1}}+\sqrt{\frac{2\log{(1/\delta)}}{n_1}}\,\middle|N_1, N_2\right]\geq 1-\delta.
    \end{equation*}
    However, by taking the~expectation w.r.t. the~decomposition $N_1$ and $N_2$ on both sides (and considering its monotonicity), we obtain
    \begin{equation*}
        \prob\left[\abs{\D(\vec{\mu}\hs; P)-\D(\widehat{\vec{\mu}}_n\hs; P)}\leq \frac{n_2}{n_1}+2\sqrt{\frac{1440 \pi e}{1-e^{-1}}}\sqrt{\frac{d+1}{n_1}}+\sqrt{\frac{2\log{(1/\delta)}}{n_1}}\right]\geq 1-\delta.
    \end{equation*}
    Now, we combine this result with inequality~\eqref{eq9}. Note that~for any two random events $A, B$ we have $1 \geq \prob(a~\cup B) = \prob(A) + \prob(B) - \prob(a~\cap B)$, which gives $\prob(A\cap B)\geq \prob(A)+\prob(B)-1$. Therefore, it holds that
    \begin{equation}\label{eq10b}
    \begin{aligned}
        \prob\left[\abs{\D(\vec{\mu}\hs; P)-\D(\widehat{\vec{\mu}}_n\hs; P)}\leq \frac{n_2}{n_1}+2\sqrt{\frac{1440 \pi e}{1-e^{-1}}}\sqrt{\frac{d+1}{n_1}}+\sqrt{\frac{2\log{(1/\delta)}}{n_1}},\right.\\
        \left.\frac{n_2}{n_1}\leq\frac{\varepsilon}{1-\varepsilon}+\frac{9}{2}\sqrt{\frac{\log(1/\delta)}{2n}}\leq 2\right]\geq 1-2\delta. 
    \end{aligned}
    \end{equation}
    Now, under the~condition of the~second random event in~\eqref{eq10b}, we can~further upper bound 
    \begin{align*}
        \left\vert \D(\vec{\mu}\hs; P) \right. & \left. -\D(\widehat{\vec{\mu}}_n\hs; P) \right\vert \leftstackrel{\eqref{eq11a}}{\leq}\frac{n_2}{n_1}+2\sqrt{\frac{1440 \pi e}{1-e^{-1}}}\sqrt{\frac{d+1}{n_1}}+\sqrt{\frac{2\log{(1/\delta)}}{n_1}}\\
        &\leftstackrel{\eqref{eq9}}{\leq}\frac{\varepsilon}{1-\varepsilon}+\frac{9}{2}\sqrt{\frac{\log(1/\delta)}{2n}}+2\sqrt{\frac{1440 \pi e}{1-e^{-1}}}\sqrt{\frac{d+1}{n_1}}+\sqrt{\frac{2\log{(1/\delta)}}{n_1}} \\
        &\leftstackrel{\eqref{eq12}}{\leq}\frac{\varepsilon}{1-\varepsilon}+\frac{9}{2}\sqrt{\frac{\log(1/\delta)}{2n}}+2\sqrt{\frac{1440 \pi e}{1-e^{-1}}}\sqrt{\frac{3(d+1)}{n}}+\sqrt{\frac{6\log{(1/\delta)}}{n}} \\
        &=\frac{\varepsilon}{1-\varepsilon}+24\sqrt{\frac{30\pi e}{1-e^{-1}}}\sqrt{\frac{d+1}{n}}+\frac{9\sqrt{2}+4\sqrt{6}}{4}\sqrt{\frac{\log{(1/\delta)}}{n}} \\
        &\leftstackrel{1\leq d}{\leq}\frac{\varepsilon}{1-\varepsilon}+24\sqrt{\frac{30\pi e}{1-e^{-1}}}\sqrt{\frac{2d}{n}}+\frac{9\sqrt{2}+4\sqrt{6}}{4}\sqrt{\frac{\log{(1/\delta)}}{n}} \\
        &= \frac{\varepsilon}{1-\varepsilon}+C_1\sqrt{\frac{d}{n}}+C_2\sqrt{\frac{\log{(1/\delta)}}{n}}.
    \end{align*}
    Ultimately, we have
    \begin{equation*}
    \begin{aligned}
     \prob\left[\abs{\D(\vec{\mu}\hs; P)-\D(\widehat{\vec{\mu}}_n\hs; P)}\leq \frac{\varepsilon}{1-\varepsilon}+C_1\sqrt{\frac{d}{n}}+C_2\sqrt{\frac{\log{(1/\delta)}}{n}},\right.\\ \left.\frac{n_2}{n_1}\leq\frac{\varepsilon}{1-\varepsilon}+\frac{9}{2}\sqrt{\frac{\log(1/\delta)}{2n}}\leq 2\right]\geq 1-2\delta. 
    \end{aligned}
    \end{equation*}
    For any random events $A, B$, we have $\prob(A\cap B)\leq \prob(A)$. Therefore, the~preceding inequality implies that
    \begin{equation*}
       \abs{\D(\vec{\mu}\hs; P)-\D(\widehat{\vec{\mu}}_n\hs; P)}\leq \frac{\varepsilon}{1-\varepsilon}+C_1\sqrt{\frac{d}{n}}+C_2\sqrt{\frac{\log{(1/\delta)}}{n}}
    \end{equation*}
    holds with probability at~least $1-2\delta$. The~proof is concluded.

    \section{Difficulties in establishing upper bounds for the~scatter halfspace median~matrix with \texorpdfstring{$\alpha\ne 2$}{alpha<2}} 

Unlike in the~situation with the~spherically symmetric distributions in Section~\ref{section: rate for elliptical}, for general $\alpha$-symmetric distributions $P \in \P{d}$, the~concentration inequality for the~\textbf{sHD} of the~scatter halfspace median~matrix $\mat{\Sigma}\hs = \sigma^2 \mat{I}$ of $P$ in Lemma~\ref{lemma: SHD rate} does not warrant a~concentration inequality for the~sample scatter halfspace median~$\widehat{\mat{\Sigma}}\hs_n$. We illustrate this in the~situation with $\alpha< 2$; for $\alpha> 2$, analogous arguments apply. The~problem with establishing rates for $\widehat{\mat{\Sigma}}\hs_n$ is due to two reasons:
    \begin{enumerate}[label=(\roman*), ref=(\roman*)]
    \item \label{problem i} The~gap in the~range of values
    \begin{equation} \label{eq: gap}
    \sigma~\, d^{1/2-1/\alpha} = \inf_{\vec{u} \in \S{d}} \frac{\sqrt{\vec{u}\T \mat{\Sigma}\hs \vec{u}}}{\norm{\vec{u}}_{\alpha}} < \sup_{\vec{u} \in \S{d}} \frac{\sqrt{\vec{u}\T \mat{\Sigma}\hs \vec{u}}}{\norm{\vec{u}}_{\alpha}} = \sigma.
    \end{equation}
    These two bounds play a~major role in the~expression~\eqref{eq: SHD for alpha} for the~\textbf{sHD} of $\mat{\Sigma}\hs$. As we will see, Lemma~\ref{lemma: SHD rate} allows us to bound only the~range of the~map $\varphi_n \colon \S{d} \to \R \colon \vec{u} \mapsto \vec{u}\T \widehat{\mat{\Sigma}}\hs_n \vec{u}$, which can~be proved to be close to the~range of $\varphi \colon \S{d} \to \R \colon \vec{u} \mapsto \vec{u}\T \mat{\Sigma}\hs \vec{u}$. The~gap in~\eqref{eq: gap}, however, does not allow us to relate the~individual values $\varphi_n(\vec{u})$ and $\varphi(\vec{u})$ as would be needed to bound the~norm of $\widehat{\mat{\Sigma}}\hs_n - \mat{\Sigma}\hs$.
    \item \label{problem ii} The~fact that~as dimension $d$ increases, the~constant $\sigma$ in~\eqref{eq: conditionSigma} grows to infinity. This means that~for obtaining a~concentration inequality valid in any dimension $d$, as in the~location case or for $\alpha= 2$, one would need to impose a~condition similar to~\ref{A3} with $\abs{t - \sigma^2} < \gamma$ with arbitrarily large $\sigma$, which is impossible due to $F$ being bounded from above. 
    \end{enumerate}

We conclude our discussion by elaborating on these two issues in more detail. To explain why~\ref{problem i} causes problems for establishing the~upper bound, consider $P\in\P{d}$ $\alpha$-symmetric with the~distribution function of its first marginal $F$. Then $\mat{\Sigma}\hs=\sigma^2\mat{I}$, where $\sigma$ is defined by~\eqref{eq: conditionSigma}. Suppose that~the~inequality
\begin{equation*}
    \abs{\SHD(\sigma^2\mat{I}; P)-\SHD(\widehat{\mat{\Sigma}}_n\hs; P)}\leq \frac{\varepsilon}{1-\varepsilon}+C_1\sqrt{\frac{d}{n}}+C_2\sqrt{\frac{\log{(1/\delta)}}{n}}\eqdef R(\delta, n, d, \varepsilon)
\end{equation*}
from Lemma~\ref{lemma: SHD rate} holds. Using the~expression of the~\textbf{sHD} for $\alpha$-symmetric distributions~\eqref{eq: SHD for alpha}, we can~deduce that~
\begin{align*}
    & \left(F(\sigma~d^{1/2-1/\alpha})-\frac{1}{2}\right)-\inf_{\norm{\vec{u}}_\alpha=1}\min\set{F\left(\sqrt{\vec{u}\T\widehat{\mat{\Sigma}}_n\hs\vec{u}}\right)-\frac{1}{2}, 1-F\left(\sqrt{\vec{u}\T\widehat{\mat{\Sigma}}_n\hs\vec{u}}\right)}\\
    &= \left(1-F(\sigma)\right)-\inf_{\norm{\vec{u}}_\alpha=1}\min\set{F\left(\sqrt{\vec{u}\T\widehat{\mat{\Sigma}}_n\hs\vec{u}}\right)-\frac{1}{2}, 1-F\left(\sqrt{\vec{u}\T\widehat{\mat{\Sigma}}_n\hs\vec{u}}\right)} \\
    & \leq R(\delta, n, d, \varepsilon)/2. 
\end{align*}
This implies that~for any $\vec{u}\in\R^d, \norm{\vec{u}}_\alpha=1$, it must hold that
    \[
    \begin{aligned}
    F(\sigma~d^{1/2-1/\alpha})-F\left(\sqrt{\vec{u}\T\widehat{\mat{\Sigma}}_n\hs\vec{u}}\right) & \leq R(\delta, n, d, \varepsilon)/2,\\
    F\left(\sqrt{\vec{u}\T\widehat{\mat{\Sigma}}_n\hs\vec{u}}\right)-F(\sigma) & \leq R(\delta, n, d, \varepsilon)/2.
    \end{aligned}
    \]
Also, recall that~$\sigma$ depends only on $F$ and $d$ and $F(\sigma~d^{1/2-1/\alpha})\leq F(\sigma)$. Combining all of this, we have that~for any $\vec{u}$ with $\norm{\vec{u}}_\alpha=1$, 
\begin{equation}\label{eq: interval of sc median}
    F\left(\sqrt{\vec{u}\T\widehat{\mat{\Sigma}}_n\hs\vec{u}}\right) \in \left[ F(\sigma~d^{1/2-1/\alpha})-R(\delta, n, d, \varepsilon)/2, F(\sigma)+R(\delta, n, d, \varepsilon)/2\right].
\end{equation}
As opposed to the~situation with $\alpha= 2$ and the~resulting bound~\eqref{eq: bound scatter 3/4}, for $\alpha\ne 2$ we see that~the~situation is fundamentally different. Instead of having a~bound on 
\begin{equation*}
    \abs{F\left( \sqrt{\vec{u}\T \widehat{\mat{\Sigma}}_n\hs \vec{u}} \right) - 3/4} = \abs{F\left( \sqrt{\vec{u}\T \widehat{\mat{\Sigma}}_n\hs \vec{u}} \right) - F\left( \sqrt{\vec{u}\T \mat{\Sigma}\hs \vec{u}} \right)}
\end{equation*} valid for all $\vec{u} \in \S{d}$, in \eqref{eq: interval of sc median} we can~bound only the~range of the~values that~$F\left( \sqrt{\vec{u}\T \widehat{\mat{\Sigma}}_n\hs \vec{u}} \right)$ must take. The~length of this range does not converge to $0$ as $n\to \infty$. From such a~crude result, bounding the~deviation $\abs{\vec{u}\T \widehat{\mat{\Sigma}}_n\hs\vec{u} - \vec{u}\T \mat{\Sigma}\hs \vec{u}}$ for individual vectors $\vec{u} \in \S{d}$ is not possible, even under a~condition guaranteeing an~appropriate growth of $F$ such as~\ref{A3}. 

In the~following example, we illustrate the~other problem~\ref{problem ii} with establishing the~upper bound for $\widehat{\mat{\Sigma}}_n\hs$.

\begin{example}
    Take $P \in \P{d}$ the~$1$-symmetric distribution with independent Cauchy marginals from Example~\ref{ex: 1}. The~distribution function of its first marginal is $F(t)=1/2+\arctan(t)/\pi$ for $t\in\R$. By Theorem~\ref{th: scatter median of symmetric}, the~only scatter halfspace median~of $P$ is $\sigma^2\mat{I}$, where $\sigma>0$ is given by
    \begin{equation*}
        \arctan(\sigma\,d^{-1/2})/\pi= 1/2 - \arctan(\sigma)/\pi = \arctan(1/\sigma)/\pi,
    \end{equation*} 
    where in the~second equality we used that~for $\sigma>0$, the~equality $\pi/2-\arctan(\sigma)=\arctan(1/\sigma)$ holds. 
    We obtain $\sigma~= d^{1/4}$, and the~unique median~matrix of $P$ is $\mat{\Sigma}\hs = \sqrt{d}\,\mat{I}$. The~maximum \textbf{sHD} is 
    \begin{equation*}
    \max_{\mat{\Sigma}\in \PD} \SHD(\mat{\Sigma}; P) = \SHD(\mat{\Sigma}\hs; P) = 2 \left( F(d^{-1/4})-1/2 \right)=\frac{2}{\pi}\arctan(d^{-1/4}),
    \end{equation*}
    which goes to $0$ is $d\to \infty$. This is the~same result as~\citet[Theorem 4.4]{Paindaveine2018} obtained by calculating the~exact \textbf{sHD} of any matrix w.r.t. $P$ and maximizing it.

    The~difficulty with our bounds~\eqref{eq: interval of sc median} is that~if $\alpha\neq 2$, then $\sigma$ depends on $d$. In our case of $\alpha= 1$ and the~Cauchy distribution, for example,~\eqref{eq: interval of sc median} rewrites into
    \begin{equation*}
    \arctan\left(\sqrt{\vec{u}\T\widehat{\mat{\Sigma}}_n\hs\vec{u}}\right) \in \left[ \arctan(d^{-1/4}) -\frac{\pi}{2} R(\delta, n, d, \varepsilon), \arctan(d^{1/4}) + \frac{\pi}{2} R(\delta, n, d, \varepsilon)\right].
    \end{equation*}
    To invert the~inequality from above into one for $\abs{\sqrt{\vec{u}\T\widehat{\mat{\Sigma}}_n\hs\vec{u}} - d^{1/4}}$, one would need to establish a~condition analogous to~\ref{A3} that~is valid uniformly among all $\sigma~= d^{1/4}$, for all dimensions $d$. That~is, of course, impossible, as the~distribution function $F$ is bounded from above.
\end{example}

\bibliographystyle{apalike}
\bibliography{bibliography}{}